\documentclass[10pt]{amsart}

\usepackage[
paper=a4paper,
text={138mm,203mm},centering
]{geometry}

\usepackage{amssymb}
\usepackage[all]{xy}
\usepackage{pb-diagram,pb-xy}
\usepackage{graphicx,color,float}
\usepackage[bookmarks]{hyperref}
\usepackage{pinlabel}

\iftrue
\makeatletter
\def\@settitle{%
  \vspace*{-20pt}
  \begin{flushleft}%
    \baselineskip14\p@\relax
    \normalfont\bfseries\LARGE
    \@title
  \end{flushleft}%
}
\def\@setauthors{%
  \begingroup
  \def\thanks{\protect\thanks@warning}%
  \trivlist
  \large \@topsep30\p@\relax
  \advance\@topsep by -\baselineskip
  \item\relax
  \author@andify\authors
  \def\\{\protect\linebreak}%
  \authors
  \ifx\@empty\contribs
  \else
    ,\penalty-3 \space \@setcontribs
    \@closetoccontribs
  \fi
  \normalfont
  \@setaddresses
  \endtrivlist
  \endgroup
}
\def\@setaddresses{\par
  \nobreak \begingroup
  \small
  \def\author##1{\nobreak\addvspace\smallskipamount}%
  \def\\{\unskip, \ignorespaces}%
  \interlinepenalty\@M
  \def\address##1##2{\begingroup
    \par\addvspace\bigskipamount\noindent
    \@ifnotempty{##1}{(\ignorespaces##1\unskip) }%
    {\ignorespaces##2}\par\endgroup}%
  \def\curraddr##1##2{\begingroup
    \@ifnotempty{##2}{\nobreak\noindent\curraddrname
      \@ifnotempty{##1}{, \ignorespaces##1\unskip}\/:\space
      ##2\par}\endgroup}%
  \def\email##1##2{\begingroup
    \@ifnotempty{##2}{\nobreak\noindent E-mail address%
      \@ifnotempty{##1}{, \ignorespaces##1\unskip}\/:\space
      \ttfamily##2\par}\endgroup}%
  \def\urladdr##1##2{\begingroup
    \def~{\char`\~}%
    \@ifnotempty{##2}{\nobreak\noindent\urladdrname
      \@ifnotempty{##1}{, \ignorespaces##1\unskip}\/:\space
      \ttfamily##2\par}\endgroup}%
  \addresses
  \endgroup
  \global\let\addresses=\@empty
}
\def\@setabstracta{%
    \ifvoid\abstractbox
  \else
    \skip@25\p@ \advance\skip@-\lastskip
    \advance\skip@-\baselineskip \vskip\skip@
    \box\abstractbox
    \prevdepth\z@ 
    \vskip-10pt
  \fi
}
\renewenvironment{abstract}{%
  \ifx\maketitle\relax
    \ClassWarning{\@classname}{Abstract should precede
      \protect\maketitle\space in AMS document classes; reported}%
  \fi
  \global\setbox\abstractbox=\vtop \bgroup
    \normalfont\small
    \list{}{\labelwidth\z@
      \leftmargin0pc \rightmargin\leftmargin
      \listparindent\normalparindent \itemindent\z@
      \parsep\z@ \@plus\p@
      
    }%
    \item[\hskip\labelsep\bfseries\abstractname.]%
}{%
  \endlist\egroup
  \ifx\@setabstract\relax \@setabstracta \fi
}
\def\section{\@startsection{section}{1}%
  \z@{-1.2\linespacing\@plus-.5\linespacing}{.8\linespacing}%
  {\normalfont\bfseries\Large}}
\def\subsection{\@startsection{subsection}{2}%
  \z@{-.8\linespacing\@plus-.3\linespacing}{.3\linespacing\@plus.2\linespacing}%
  {\normalfont\bfseries}}
\def\subsubsection{\@startsection{subsection}{3}%
  \z@{.7\linespacing\@plus.2\linespacing}{-1.5ex}%
  {\normalfont\itshape}}
\def\@secnumfont{\bfseries}
\makeatother
\fi 

\def\to{\mathchoice{\longrightarrow}{\rightarrow}{\rightarrow}{\rightarrow}}
\makeatletter
\newcommand{\shortxra}[2][]{\ext@arrow 0359\rightarrowfill@{#1}{#2}}
\def\longrightarrowfill@{\arrowfill@\relbar\relbar\longrightarrow}
\newcommand{\longxra}[2][]{\ext@arrow 0359\longrightarrowfill@{#1}{#2}}
\renewcommand{\xrightarrow}[2][]{\mathchoice{\longxra[#1]{#2}}%
  {\shortxra[#1]{#2}}{\shortxra[#1]{#2}}{\shortxra[#1]{#2}}}
\makeatother

\def\otimesover#1{\mathbin{\mathop{\otimes}_{#1}}}

\makeatletter
\def\Nopagebreak{\@nobreaktrue\nopagebreak}
\makeatother

\theoremstyle{plain}
\newtheorem{theorem}{Theorem}[section]
\newtheorem{proposition}[theorem]{Proposition}
\newtheorem{corollary}[theorem]{Corollary}
\newtheorem{lemma}[theorem]{Lemma}

\newtheorem{assertion}{Assertion}

\theoremstyle{definition}
\newtheorem{definition}[theorem]{Definition}

\newtheorem{example}[theorem]{Example}
\newtheorem{remark}[theorem]{Remark}

\def\Z{\mathbb{Z}}
\def\Q{\mathbb{Q}}
\def\R{\mathbb{R}}
\def\C{\mathbb{C}}

\def\N{\mathcal{N}}

\def\F{\mathcal{F}}

\def\K{\mathcal{K}}

\def\cC{\mathcal{C}}

\def\cP{\mathcal{P}}

\def\Ker{\operatorname{Ker}}
\def\Coker{\operatorname{Coker}}
\def\Im{\operatorname{Im}}
\def\Hom{\operatorname{Hom}}
\def\Tor{\operatorname{Tor}}
\def\Ext{\operatorname{Ext}}
\def\sign{\operatorname{sign}}
\def\rank{\operatorname{rank}}

\def\inte{\operatorname{int}}

\def\Bl{B\ell}
\def\lk{\operatorname{lk}}
\def\Arf{\operatorname{Arf}}

\def\ltr{\operatorname{tr}^{(2)}}
\def\ldim{\dim^{(2)}}
\def\lsign{\sign^{(2)}}
\def\lt{\ell^2}
\def\Lt{L^2}
\def\rhot{\rho^{(2)}}

\def\mathbinover#1#2{\mathbin{\mathop{#1}\limits_{#2}}}
\def\amalgover#1{\mathbinover{\amalg}{#1}}

\def\emptystr{}
\newcommand{\mkc}[2][]{\begin{color}{red}#2%
  \def\tempstr{#1}%
  \ifx\tempstr\emptystr \else\textsf{\SMALL\ \raise.7ex\hbox{[\tempstr]}}\fi
\end{color}}

\begin{document}

\title%
{Amenable $L^2$-theoretic methods and knot concordance}

\author{Jae Choon Cha}

\address{Department of Mathematics and PMI\\
  POSTECH \\
  Pohang 790--784\\
  Republic of Korea}

\email{jccha@postech.ac.kr}

\def\subjclassname{\textup{2010} Mathematics Subject Classification}
\expandafter\let\csname subjclassname@1991\endcsname=\subjclassname
\expandafter\let\csname subjclassname@2000\endcsname=\subjclassname


\begin{abstract}
  We introduce new obstructions to topological knot concordance.
  These are obtained from amenable groups in Strebel's class, possibly
  with torsion, using a recently suggested $L^2$-theoretic method due
  to Orr and the author.
  Concerning $(h)$-solvable knots which are defined in terms of
  certain Whitney towers of height~$h$ in bounding 4-manifolds, we use
  the obstructions to reveal new structure in the knot concordance
  group not detected by prior known invariants: for any $n>1$ there
  are $(n)$-solvable knots which are not $(n.5)$-solvable (and
  therefore not slice) but have vanishing Cochran-Orr-Teichner
  $L^2$-signature obstructions as well as Levine algebraic
  obstructions and Casson-Gordon invariants.
\end{abstract}

\maketitle

\section{Introduction}

In this paper we introduce new obstructions to knots being
topologically slice and concordant, and study the structure of the
knot concordance group using these.  Recall that two knots in $S^3$
are (topologically) \emph{concordant} if there is a locally flat
proper embedding of an annulus in $S^3\times[0,1]$ which restricts to
the given knots on the boundary components.  A knot is called
(topologically) \emph{slice} if it is concordant to the trivial knot,
or equivalently, if it bounds a locally flat 2-disk properly embedded
in the 4-ball.  The concordance classes of knots in $S^3$ form an
abelian group $\cC$ under connected sum, which is called the
\emph{knot concordance group}.  Slice knots represent the identity
in~$\cC$.  Recall that obstructions to being topologically slice or
concordant, particularly those we discuss in this paper, are also
obstructions in the smooth category.

Since the beginning due to Fox and Milnor in the 50's, significant
progress has been made in the study of the knot concordance group.
After the landmarks of Levine~\cite{Levine:1969-1,Levine:1969-2} and
Casson-Gordon~\cite{Casson-Gordon:1986-1,Casson-Gordon:1978-1}, the
latest breakthrough which opened up a new direction was made in the
work of Cochran, Orr, and Teichner~\cite{Cochran-Orr-Teichner:1999-1,
  Cochran-Orr-Teichner:2002-1}.  They developed theory of obstructions
to being slice, which detects non-slice examples for which prior
invariants of Levine and Casson-Gordon vanish.  Their obstructions are
$L^2$-signature defects of bounding 4-manifolds, or equivalently, the
von Neumann-Cheeger-Gromov $\rho$-invariants of 3-manifolds,
associated to certain homomorphisms of the fundamental group into
\emph{poly-torsion-free-abelian (PTFA)} groups.  We recall that a
group $G$ is PTFA if it admits a subnormal series $G=G_0 \rhd
\def\ignoreme{G_1 \rhd} \cdots \rhd G_r=\{e\}$ with each $G_i/G_{i+1}$
torsion-free abelian.  Subsequent to their
work~\cite{Cochran-Orr-Teichner:1999-1, Cochran-Orr-Teichner:2002-1},
many interesting new results on concordance, homology cobordism and
related topics have been obtained using the PTFA $L^2$-signatures by
several authors.

As the main result of this paper, we obtain new obstructions which
detect many elements in the knot concordance group that are not
distinguished by any previously known obstructions including the PTFA
$L^2$-signatures and the invariants of Levine and Casson-Gordon.

To give a more precise description of our results, we recall a
framework of recent systematic study of the knot concordance group
under which known obstructions are best understood.  In
\cite{Cochran-Orr-Teichner:1999-1}, Cochran, Orr, and Teichner defined
a geometrically defined filtration
\[
0\subset \cdots \subset \F_{n.5} \subset \F_{n} \subset \cdots \subset \F_1
\subset \F_{0.5} \subset \F_0 \subset \cC
\]
of the knot concordance group~$\cC$ indexed by half integers, which is
closely related to the theory of topological 4-manifolds via Whitney
towers and gropes.  (For definitions and related discussions, readers
are referred to \cite[Sections~7,~8]{Cochran-Orr-Teichner:1999-1}.)
Loosely speaking, $\F_n$ is the subgroup of the classes of
$(n)$-solvable knots, where a knot $K$ is defined to be
\emph{$(n)$-solvable} if its zero-surgery manifold $M(K)$ bounds a
spin 4-manifold $W$ which has $H_1(W)\cong \Z$ generated by a meridian
and admits a Whitney tower of height $n$ for a spherical lagrangian of
the intersection form of $W$ in the sense of \cite
{Cochran-Orr-Teichner:1999-1}.  Such $W$ is called an
\emph{$(n)$-solution} for~$K$.  An $(n.5)$-solvable knot and its
$(n.5)$-solution are defined similarly as refinements between level
$n$ and~$n+1$.  It is well known that a slice knot is $(h)$-solvable
for any $h$, with the slice disk exterior as an $(h)$-solution.  The
obstructions from Levine's algebraic invariants and Casson-Gordon
invariants vanish if a knot is $(0.5)$- and $(1.5)$-solvable,
respectively.

In~\cite{Cochran-Orr-Teichner:1999-1}, Cochran-Orr-Teichner introduced
PTFA $L^2$-signature obstructions to being $(n.5)$-solvable and
initiated the study of concordance of highly solvable knots for which
Levine and Casson-Gordon invariants vanish:

\begin{theorem}[{\cite[Theorem~4.2]{Cochran-Orr-Teichner:1999-1}}]
  \label{theorem:COT-obstruction}
  Suppose $K$ is an $(n.5)$-solvable knot.  Suppose $G$ is a PTFA
  group, $G^{(n+1)}=\{e\}$, and $\phi\colon \pi_1 M(K) \to G$ is a
  homomorphism extending to an $(n.5)$-solution~$W$.  Then
  $\rhot(M(K),\phi)=0$.
\end{theorem}

Here $G^{(n+1)}$ is the derived subgroup of $G$ defined inductively by
$G^{(0)}=G$, $G^{(n+1)}=[G^{(n)},G^{(n)}]$, and $\rhot(M,\phi)$ denotes
the von Neumann-Cheeger-Gromov $\rho$-invariant, which is equal to the
$L^2$-signature defect of a bounding 4-manifold.  (For its definition,
see, e.g.,
Section~\ref{subsection:proof-of-theorem-slice-obstruction}.)

Since a PTFA group $G$ satisfies $G^{(n)}=\{e\}$ for some $n$ and a
slice disk exterior of a knot $K$ is an $(h)$-solution for any $h$, one
has the following consequence of
Theorem~\ref{theorem:COT-obstruction}: \emph{For any slice knot $K$,
  if $G$ is PTFA and $\phi\colon \pi_1 M(K) \to G$ extends to a slice
  disk exterior, then $\rhot(M,\phi)=0$.}

We remark that
for \emph{link} concordance, there are other recent techniques
revealing deep information that is invisible via $L^2$-signatures.
For example, the $L$-group valued Hirzebruch-type invariants from
iterated $p$-covers \cite{Cha:2007-1,Cha:2007-2}, covering link
calculus techniques
\cite{Cha-Livingston-Ruberman:2006-1,Cha-Kim:2008-1,VanCott:2009-01},
and twisted torsion invariants \cite{Cha-Friedl:2010-01} are known to
detect interesting highly solvable non-slice links for which
$L^2$-signatures are not effective.

However, for \emph{knots}, the PTFA $L^2$-signatures have been known
as the only useful tool to distinguish highly solvable knots up to
concordance, particularly for those with vanishing Casson-Gordon and
Levine obstructions.  Roughly speaking, this interesting subtlety
peculiar to knots is related to the ``size'' of the fundamental
group---for knots and 3-manifolds with the first $\Z_p$-Betti number
$b_1(M;\Z_p)\le 1$ for any prime~$p$, the only previously known
nonabelian coverings from which one can extract information on
concordance and homology cobordism are PTFA covers and certain
metabelian covers considered by Casson-Gordon.  Indeed all known
results on knot concordance beyond Levine and Casson-Gordon invariants
essentially depend on Theorem~\ref{theorem:COT-obstruction}.  (See
also the remarkable works of Cochran-Teichner
\cite{Cochran-Teichner:2003-1} and Cochran-Harvey-Leidy
\cite{Cochran-Harvey-Leidy:2009-1, Cochran-Harvey-Leidy:2008-2,
  Cochran-Harvey-Leidy:2009-02, Cochran-Harvey-Leidy:2009-03}.)


Recently, in \cite{Cha-Orr:2009-01}, Orr and the author have presented
a new $L^2$-theoretic method relating homological properties and
$L^2$-invariants, using a result of L\"uck.  Their approach differs
fundamentally from the prior PTFA techniques that are mostly
algebraic.  The approach in \cite{Cha-Orr:2009-01} extends the use of
$L^2$-signatures to a significantly larger class of groups, namely, the
class of amenable groups lying in Strebel's class $D(R)$.  (Here $R$
is a commutative ring; see
Section~\ref{subsection:L2-homology-and-amenable-groups-in-D(R)} for
definitions and related discussions.)  This class contains several
interesting infinite/finite non-torsion-free groups, and subsumes PTFA
groups.  For example, see Lemma~\ref{lemma:poly-mixed-groups}.
In \cite{Cha-Orr:2009-01} they proved the homology cobordism
invariance of $L^2$-signature defects and $L^2$-Betti numbers
associated to amenable groups in $D(R)$ and gave several examples and
applications.

In this paper, we further develop the $L^2$-theoretic method initiated
in \cite{Cha-Orr:2009-01} to study the structure of the knot
concordance group beyond the information from PTFA $L^2$-signatures
and invariants of Casson-Gordon and Levine.  First we give new
obstructions to a knot being slice and to being $(n.5)$-solvable.

\begin{theorem}
  \label{theorem:intro-obstruction-slice}
  Suppose $K$ is a slice knot and $G$ is an amenable group lying in
  Strebel's class $D(R)$ for some $R$.  If $\phi\colon \pi_1M(K) \to
  G$ is a homomorphism extending to a slice disk exterior for $K$,
  then $\rhot(M(K),\phi)=0$.
\end{theorem}

\begin{theorem}
  \label{theorem:intro-obstruction-n.5-solvable}
  Suppose $K$ is an $(n.5)$-solvable knot.  Suppose $G$ is an amenable
  group lying in Strebel's class $D(R)$ where $R$ is $\Q$ or $\Z_p$,
  $G^{(n+1)}=\{e\}$, and $\phi\colon \pi_1M(K) \to G$ is a
  homomorphism which extends to an $(n.5)$-solution for $K$ and sends
  a meridian to an infinite order element in~$G$.  Then
  $\rhot(M(K),\phi)=0$.
\end{theorem}

We remark that the meridian condition in
Theorem~\ref{theorem:intro-obstruction-n.5-solvable} is not a severe
restriction since it is satisfied in most cases.
It can be seen that Cochran-Orr-Teichner's
Theorem~\ref{theorem:COT-obstruction} is a consequence of our
Theorem~\ref{theorem:intro-obstruction-n.5-solvable}.

An important aspect of the new slice obstruction
(Theorem~\ref{theorem:intro-obstruction-slice}) is that it is
\emph{not} a consequence of the obstruction to being $(n.5)$-solvable
(Theorem~\ref{theorem:intro-obstruction-n.5-solvable}), contrary to
the PTFA case.  The group $G$ in
Theorem~\ref{theorem:intro-obstruction-slice} may be
\emph{non-solvable}.
This offers a very interesting potential to detect non-slice knots
which look like slice knots on invariants from solvable groups.  This
will be addressed in a subsequent paper.

As an application of
Theorem~\ref{theorem:intro-obstruction-n.5-solvable}, for each $n$, we
produce a large family of $(n)$-solvable knots which are not
$(n.5)$-solvable but not detected by the PTFA $L^2$-signatures:

\begin{theorem}
  \label{theorem:intro-linearly-independent-examples}
  For any $n$, there are infinitely many $(n)$-solvable knots $J^i$
  $(i=1,2,\ldots)$ satisfying the following: any linear combination
  $\#_i\, a_iJ^i$ under connected sum is an $(n)$-solvable knot with
  vanishing PTFA $L^2$-signature obstructions, but whenever $a_i\ne 0$
  for some~$i$, $\#_i\, a_iJ^i$ is not $(n.5)$-solvable.  Consequently
  the $J^i$ generate an infinite rank subgroup in $\F_n/\F_{n.5}$
  which is invisible via PTFA $L^2$-signature obstructions.
\end{theorem}

Here we say that $J$ is \emph{an $(n)$-solvable knot $J$ with
  vanishing PTFA $L^2$-signature obstructions} if there is an
$(n)$-solution $W$ for $J$ such that for any PTFA group $G$ and for
any $\phi\colon \pi_1 M(J) \to G$ extending to $W$,
$\rhot(M(J),\phi)=0$.  (We do not require $G^{(n+1)}=\{e\}$.)  It
turns out that such knots $J$ form a subgroup $\mathcal{V}_n$ of
$\F_n$ (see
Definition~\ref{definition:knots-with-vanishing-PTFA-signature-obstructions}
and Proposition~\ref{proposition:vanishing-PTFA-subgroup}).  Obviously
$\F_{n.5}\subset \mathcal{V}_n \subset \F_n$.
Theorem~\ref{theorem:intro-linearly-independent-examples} says that
$\mathcal{V}_n/\F_{n.5}$ has infinite rank.

For knots in $\mathcal{V}_n$, the conclusion of the
Cochran-Orr-Teichner theorem obstructing being $(n.5)$-solvable
(Theorem~\ref{theorem:COT-obstruction}) holds even for some
$(n)$-solution~$W$ which is not necessarily an $(n.5)$-solution.
Consequently, all the prior techniques using the PTFA obstructions
(e.g., \cite{Cochran-Orr-Teichner:1999-1, Cochran-Orr-Teichner:2002-1,
  Cochran-Teichner:2003-1, Cochran-Kim:2004-1,
  Cochran-Harvey-Leidy:2008-2, Cochran-Harvey-Leidy:2009-1,
  Cochran-Harvey-Leidy:2009-02, Cochran-Harvey-Leidy:2009-03}) fail to
distinguish any knots in $\mathcal{V}_n$, particularly our examples in
Theorem~\ref{theorem:intro-linearly-independent-examples}, from
$(n.5)$-solvable knots up to concordance.  Obviously the invariants of
Levine and Casson-Gordon also vanish for knots in $\mathcal{V}_n$ for
$n\ge 2$.


We remark that our coefficient system used in the proof of
Theorem~\ref{theorem:intro-linearly-independent-examples} can be
viewed as a \emph{higher-order generalization} of the Casson-Gordon
metabelian setup.  Recall that
Casson-Gordon~\cite{Casson-Gordon:1986-1, Casson-Gordon:1978-1}
extract invariants from a $p$-torsion abelian cover of the infinite
cyclic cover of the zero-surgery manifold~$M(K)$.  Our coefficient
system corresponds to a tower of covers
\[
M_{n+1} \xrightarrow{p_n} M_{n} \xrightarrow{p_{n-1}} \cdots
\xrightarrow{p_1} M_1 \xrightarrow{p_0} M_0 = \text{zero-surgery
  manifold $M(K)$}
\]
where $p_0$ is the infinite cyclic cover, $p_1,\ldots,p_{n-1}$ are
torsion-free abelian covers, and $p_n$ is a $p$-torsion cover.  This
iterated covering for knots may also be compared with the iterated
$p$-cover construction in~\cite{Cha:2007-1,Cha:2007-2} which is useful
for links.


To construct such coefficient systems that factors through a given
solution, we need additional new ingredients in the proof of
Theorem~\ref{theorem:intro-linearly-independent-examples}.  First we
define a \emph{mixed-coefficient commutator series} of a group, which
generalizes the ordinary and rational derived series (see
Section~\ref{subsection:mixed-coefficient-commutator-series}).  Also,
we introduce a \emph{modulo $p$ version of the noncommutative
  higher-order Blanchfield pairing} (see
Section~\ref{section:mod-p-Blanchfield-duality}).  We combine known
ideas of PTFA techniques of Cochran-Harvey-Leidy
\cite{Cochran-Harvey-Leidy:2009-1} with these new tools in order to
construct and analyze our coefficient systems which is into certain
infinite amenable groups lying in Strebel's class with torsion.

The paper is organized as follows.  In
Section~\ref{section:slice-obstruction}, we give a proof of our slice
obstruction (Theorem~\ref{theorem:intro-obstruction-slice}), using
results of~\cite{Cha-Orr:2009-01}.  In
Section~\ref{section:obstruction-to-solvability}, we develop further
$L^2$-theoretic techniques for amenable groups in Strebel's class and
generalize some results in~\cite{Cha-Orr:2009-01}, and using these we
give a proof of our obstruction to being $(n.5)$-solvable
(Theorem~\ref{theorem:intro-obstruction-n.5-solvable}).  In
Section~\ref{section:examples-and-computation}, we construct examples
and prove Theorem~\ref{theorem:intro-linearly-independent-examples}.
In Section~\ref{section:mod-p-Blanchfield-duality}, we introduce the
modulo $p$ noncommutative Blanchfield pairing and discuss its
properties.

In this paper, manifolds are assumed to be topological and oriented,
and submanifolds are assumed to be locally flat.

\subsubsection*{Acknowledgements}
\label{subsubsection:acknowledgements}

The author owes many debts to Kent Orr, the coauthor
of~\cite{Cha-Orr:2009-01} on which this work is based.  This work was
supported by National Research Foundation of Korea (NRF) grants No.\
2009--0094069 and 2007--0054656 funded by the Korean government
(MEST).

\section{Slice knots and $L^2$-signatures over amenable groups in
  Strebel's class}
\label{section:slice-obstruction}

In this section we prove our new obstruction to being a
slice knot,
%
%
namely Theorem~\ref{theorem:intro-obstruction-slice}.

\subsection{$L^2$-homology and amenable groups in Strebel's class}
\label{subsection:L2-homology-and-amenable-groups-in-D(R)}

We start by recalling necessary results in~\cite{Cha-Orr:2009-01} on
amenable groups lying in Strebel's class~$D(R)$.  Recall that a group
$G$ is \emph{amenable} if $G$ admits a finitely-additive measure which
is invariant under the left multiplication.  For many other equivalent
definitions and further discussions, the reader is referred to,
e.g.,~\cite{Paterson:1988-1}.  Following \cite{Strebel:1974-1}, for a
commutative ring $R$ with unity, a group $G$ is said to be in $D(R)$
if a homomorphism $\alpha\colon P\to Q$ between projective
$RG$-modules is injective whenever $1_R\otimes_{RG}\alpha\colon
R\otimes_{RG}P \to R\otimes_{RG}Q$ is injective.  Here $R$ is viewed
as an $(R,RG)$-bimodule with trivial $G$-action.  For more discussions
on amenable groups lying in $D(R)$, the reader is referred to, e.g.,
\cite[Section~6]{Cha-Orr:2009-01}.

For the purpose of this paper, it suffices to use the following
description of a useful class of amenable groups in $D(\Z_p)$ which
are \emph{not torsion-free} in general:

\begin{lemma}[{\cite[Lemma~6.8]{Cha-Orr:2009-01}}]
  \label{lemma:poly-mixed-groups}
  Suppose $p$ is a fixed prime and $\Gamma$ is a group with a
  subnormal series
  \[
  \{e\} = \Gamma_n \subset \cdots \subset \Gamma_1 \subset \Gamma_0 =
  \Gamma
  \]
  such that for any $i$, $\Gamma_i/\Gamma_{i+1}$ is abelian and has no
  torsion coprime to $p$, that is, the order of an element in
  $\Gamma_i/\Gamma_{i+1}$ is either a power of $p$ or infinite.  Then
  $\Gamma$ is amenable and in $D(\Z_p)$.
\end{lemma}

The fact that PTFA groups are amenable and in~$D(\Z_p)$ can be derived
as a consequence of Lemma~\ref{lemma:poly-mixed-groups}.  In fact, a
PTFA group is in $D(R)$ for any~$R$~\cite{Strebel:1974-1}.


For a discrete countable group $G$, we denote by $\N G$ the
\emph{group von Neumann algebra} of~$G$ (e.g.,
see~\cite[Chapter~1]{Lueck:2002-1}).  $\C G$ is a subalgebra of $\N
G$, and $\N G$ is endowed with an involution $a\to a^*$ which is
induced by the inversion $g\to g^{-1}$ in~$G$.  For any module $A$
over $\N G$, the \emph{$L^2$-dimension} $\ldim A \in [0,\infty]$ is
defined as in, e.g., \cite[Chapter~6]{Lueck:2002-1}.  (For more
discussions on the $L^2$-dimension, see
Section~\ref{subsection:L2-dimension}.)

The following result due to Cha-Orr is a key ingredient which controls
the $L^2$-dimension of $L^2$-coefficient homology:

\begin{theorem}[{\cite[Theorem~6.6]{Cha-Orr:2009-01}}]
  \label{theorem:cha-orr-local-property-NG}
  Suppose $G$ is an amenable group in $D(R)$, and $C_*$ is a finitely
  generated free chain complex over~$\Z G$.  If $H_i(R \otimes_{\Z G}
  C_*)=0$ for $i\le n$, then $\ldim H_*(\N G\otimes_{\Z G} C_*)=0$ for
  $i\le n$.
\end{theorem}

\subsection{Proof of
  Theorem~\ref{theorem:intro-obstruction-slice}}
\label{subsection:proof-of-theorem-slice-obstruction}

We briefly recall the definition of the von Neumann-Cheeger-Gromov
$\rho$-invariant of a closed $3$-manifold $M$, viewing it as an
$L^2$-signature defect.  Given $\phi\colon \pi_1(M) \to G$, by
replacing $G$ with a larger group into which $G$ injects, we may
assume that there is a compact $4$-manifold $W$ with $\partial W = M$
over~$G$, by an argument of Weinberger.  Let
\[
\lambda\colon H_2(W;\N G) \times H_2(W;\N G) \to \N G
\]
be the $\N G$-coefficient intersection form of~$W$.  This is a
hermitian form over~$\N G$, so that the $L^2$-signature $\lsign
\lambda\in \R$ is defined.  Now the \emph{$L^2$-signature} of $W$ is
defined by $\lsign_G(W) = \lsign \lambda$.  For more detailed
discussions on the $L^2$-signature, see
Section~\ref{subsubsection:L2-metabolizer}.  (See also, e.g.,
\cite{Cochran-Orr-Teichner:1999-1, Lueck-Schick:2003-1,
  Chang-Weinberger:2003-1, Harvey:2006-1, Cha:2006-1,
  Cha-Orr:2009-01}.)


The von Neumann-Cheeger-Gromov $\rho$-invariant of $(M,\phi)$ is
defined to be the signature defect
\[
\rhot(M,\phi) = \lsign_G(W)-\sign(W)
\]
where $\sign(W)$ is the ordinary signature of~$W$.  $\rhot(M,\phi)$~is
known to be well-defined.

We are now ready to prove
Theorem~\ref{theorem:intro-obstruction-slice}: \emph{ Suppose $K$ is a
  slice knot and $G$ is an amenable group lying in Strebel's class
  $D(R)$ for some $R$.  If $\phi\colon \pi_1M(K) \to G$ is a
  homomorphism extending to a slice disk exterior for $K$, then
  $\rhot(M(K),\phi)=0$.}

\begin{proof}[Proof of Theorem~\ref{theorem:intro-obstruction-slice}]
  Let $W$ be a slice disk exterior for $K$, and suppose the given
  homomorphism $\phi\colon \pi_1 M(K) \to G$ factors
  through~$\pi_1(W)$.  Then, we have
  $\rhot(M(K),\phi)=\lsign_G(W)-\sign(W)$.  By Alexander duality,
  $H_2(W;R)=0$ and $H_1(W;R)\cong H_1(M(K);R)$ and thus
  $H_i(W,M(K);R)=0$ for any $R$ and $i\le 2$.  By
  Theorem~\ref{theorem:cha-orr-local-property-NG}, it follows that
  $\ldim H_2(W,M(K);\N G)=0$.  Since the non-singular part of the $\N
  G$-valued intersection form of $W$ is supported by the image of
  $H_2(W;\N G) \to H_2(W,M(K);\N G)$, it follows that $\lsign_G(W)=0$.
  Since $H_2(W,M(K))=0$, the ordinary signature $\sign(W)$ vanishes.
  Therefore $\rhot(M(K),\phi)=0$.
\end{proof}

We remark that the above proof actually shows that the
$\rho$-invariant gives us an obstruction to bounding a slice disk in a
homology 4-ball.


\section{Obstructions to being $(n.5)$-solvable}
\label{section:obstruction-to-solvability}

In this section we prove our obstruction to being an $(n.5)$-solvable
knot.

\begin{definition}
  \label{definition:solvable-knots}
  Suppose $M$ is a closed 3-manifold and $W$ is a compact
  4-manifold with boundary $M$.  Let $\pi=\pi_1(W)$.  
  \begin{enumerate}
  \item $W$ is called an \emph{integral (n)-solution} for $M$ if the
    inclusion induces $H_1(M)\cong H_1(W)$ and there exist elements
    $x_1,\ldots,x_r, y_1,\ldots, y_r\in H_2(W;\Z[\pi/\pi^{(n)}])$]
    such that $2r=\rank_\Z H_2(W;\Z)$ and the intersection form
    \[
    \lambda_n\colon H_2(W;\Z[\pi/\pi^{(n)}])\times
    H_2(W;\Z[\pi/\pi^{(n)}]) \to \Z[\pi/\pi^{(n)}]
    \]
    satisfies $\lambda_n(x_i,x_j)=0$ and $\lambda_n(x_i,y_j)=\delta_{ij}$
    (Kronecker symbol) for any $i$,~$j$.
  \item $W$ is an \emph{integral $(n.5)$-solution} of $M$ if (1) is
    satisfied for some $x_i$ and $y_j$, and there exist lifts $\tilde
    x_i\in H_2(W;\Z[\pi/\pi^{(n+1)}])$ of the $x_i$ such that
    $\lambda_{n+1}(\tilde x_i,\tilde x_j)=0$ for any $i$,~$j$.
  \end{enumerate}
  We call the collections $\{x_i\}$ and $\{\tilde x_i\}$ an
  \emph{$(n)$-lagrangian} and \emph{$(n+1)$-lagrangian}, respectively,
  and call $\{y_i\}$ their \emph{$(n)$-dual}, respectively.  If the
  zero-surgery manifold $M(K)$ of a knot $K$ has an integral
  $(h)$-solution $W$, we say that $K$ is \emph{integrally
    $(h)$-solvable}.
\end{definition}

We remark that if a knot $K$ is $(h)$-solvable in the sense
of~\cite{Cochran-Orr-Teichner:1999-1}, then $K$ is integrally
$(h)$-solvable as in Definition~\ref{definition:solvable-knots}.


The main result of this section is the following:

\begin{theorem}
  \label{theorem:n.5-solvable-obstruction-amenable-in-D(R)}
  Suppose a knot $K$ is integrally $(n.5)$-solvable, $G$ is an
  amenable group lying in $D(R)$ for $R=\Z_p$ or $\Q$,
  $G^{(n+1)}=\{e\}$, and $\phi\colon \pi_1 M(K) \to G$ is a
  homomorphism which extends to an integral $(n.5)$-solution for
  $M(K)$ and sends a meridian to an infinite order element in~$G$.
  Then $\rhot(M(K),\phi)=0$.
\end{theorem}

We note that the meridian condition holds, for example, when the
abelianization map of $\pi_1(M_K)$ factors through~$\phi$.
Theorem~\ref{theorem:intro-obstruction-n.5-solvable} in the
introduction is an immediate consequence of
Theorem~\ref{theorem:n.5-solvable-obstruction-amenable-in-D(R)}.

Before giving a proof of
Theorem~\ref{theorem:n.5-solvable-obstruction-amenable-in-D(R)}, we
discuss some basic preliminaries and establish further results on the
$L^2$-homology for amenable groups in $D(R)$, in
Sections~\ref{subsection:L2-dimension}
and~\ref{subsection:more-on-L2-homology-and-D(R)}.

\subsection{$L^2$-dimension, homology, and hermitian forms}
\label{subsection:L2-dimension}

In this subsection we discuss rapidly necessary facts on $L^2$-theory,
under the theme that the group von Neumann algebra $\N G$ behaves as a
semisimple ring to the eyes of the $L^2$-dimension.  Particularly,
several standard properties of the ordinary complex dimension of
vector spaces, dual spaces, homology are shared by the
$L^2$-dimension.

We start by reviewing the definition of the group von Neumann algebra
$\N G$ and $L^2$-dimension of modules over~$\N G$.

\subsubsection*{Definitions}

For a discrete countable group $G$, let $\lt G$ be the Hilbert space
of formal expressions $\sum_{g\in G} z_g\cdot g$ ($z_g\in \C$) which
are square summable, i.e., $\sum_{g\in G} |z_g|^2 < \infty$.  The
inner product on $\lt G$ is given by $\big\langle \sum z_g\cdot g,
\sum w_g\cdot g \big\rangle =\sum \overline{z_g} w_g$.  Let
$\mathcal{B}(\lt G)$ be the algebra of bounded operators on~$\lt G$.
(As a convention, operators act on the left.)  We can view $\C G$ as a
subalgebra of $\mathcal{B}(\lt G)$ via multiplication on the left.
The \emph{group von Neumann algebra} $\N G$ of $G$ is defined to be
the pointwise closure of $\C G$ in $\mathcal{B}(\lt G)$.  The
\emph{von Neumann trace} of $a\in \N G$ is defined by $\ltr a =
\langle a(e), e\rangle$ where $e\in G$ is the identity.  This extends
to $n\times n$ matrices $B=(b_{ij})$ over $\N G$ by $\ltr B = \sum_i
\ltr b_{ii}$.

Now we can describe the $L^2$-dimension function
\[
\ldim\colon \{ \text{all $\N G$-modules} \} \to [0,\infty]
\]
as follows.  For a finitely generated projective module $P$ over $\N
G$, there is a projection $(\N G)^n \to (\N G)^n$ whose image is
isomorphic to~$P$, that is, the associated $n\times n$ matrix, say
$B=(b_{ij})$, has row space~$P$.  Then $\ldim P$ is defined to be
$\ltr B$.  It is known that $\ldim P$ is well-defined and that one can
choose a hermitian (i.e., $B^*=(b^*_{ji})$ equals~$B$) projection
matrix $B$ with row space~$P$.  (See \cite[Chapter~6]{Lueck:2002-1}.)
For an arbitrary module $A$ over $\N G$, define $\ldim A=\sup_P \ldim
P$ where $P$ runs over finitely generated projective submodules
in~$A$.

We will frequently use the following basic properties of the
$\Lt$-dimension function.  For proofs, see
\cite[Chapter~6]{Lueck:2002-1}.

\begin{enumerate}
\item $\ldim \N G=1$ and $\ldim 0 = 0$.
\item If $0\to A'\to A \to A'' \to 0$ is exact, then $\ldim A = \ldim
  A' + \ldim A''$.  In particular, if $A'$ is a homomorphic image of a
  submodule of $A$, then $\ldim A' \le \ldim A$.
\end{enumerate}

Here we adopt the usual convention $\infty+d=\infty$ for any $d\in
[0,\infty]$.  A consequence is that $\ldim M<\infty$ if $M$ is
a quotient of a submodule of a finitely generated $\N G$-module.

The next lemma is a key result due to L\"uck, which says that a
finitely generated module over $\N G$ can be ``approximated'' by a
projective summand with the same $L^2$-dimension.

\begin{lemma}[{\cite[p.\ 239]{Lueck:2002-1}}]
  \label{lemma:NG-projective-part}
  For a finitely generated $\N G$-module $A$, let
  \[
  T(A)=\{x\in A \mid f(x)=0 \text{ for any }f\colon A\rightarrow \N
  G\}
  \]
  and $P(A)=A/T(A)$.  Then $\ldim T(A)=0$ and $P(A)$ is finitely
  generated and projective over $\N G$.  Consequently, $A\cong
  T(A)\oplus P(A)$ and $\ldim A=\ldim P(A)$.
\end{lemma}

We remark that $T(-)$ and $P(-)$ are functorial.  From
Lemma~\ref{lemma:NG-projective-part}, it follows easily that a
finitely generated $\N G$-module $A$ is projective if and only if
$T(A)=0$.

\subsubsection*{$L^2$-dimension of dual modules}

For an $\N G$-module $A$, its \emph{dual} $A^*=\Hom_{\N G}(A,\N G)$ is
an $\N G$-module under the scalar multiplication $(r\cdot f)(x)=f(x)
r^*$ for $r\in \N G$, $f\in A^*$, $x\in A$.

%

\begin{lemma}
  \label{lemma:l2-dim-dual}
  For any finitely generated $\N G$-module $A$, $T(A)^*=0=T(A^*)$,
  $P(A)^*=A^*$, and $\ldim A^* = \ldim A$.
\end{lemma}

\begin{proof}
  First we show $T(A)^*=0$.  Since $T(A)$ is a summand, any $f\colon
  T(A) \to \N G$ extends to $A$.  Therefore, $f(x)=0$ for any $x\in
  T(A)$ by definition.

  Write the given module as $A=T(A)\oplus P(A)$.  We have
  $A^*=T(A)^*\oplus P(A)^* = P(A)^*$.  Since $P(A)^*$ is projective,
  $T(A^*)=0$.  Now it suffices to show that $\ldim P^* = \ldim P$ for
  a finitely generated projective module~$P$.  Choose an $n\times n$
  hermitian projection matrix $B$ with row space $P$.  Applying
  $\Hom(-,\N G)$ to the associated map $(\N G)^n \to (\N G)^n$ and
  considering the dual basis, it is easily seen that the row space of
  $B^*$ is isomorphic to~$P^*$.  Therefore $\ldim P^* = \ltr B^* =
  \ltr B = \ldim P$.
  %
\end{proof}

\subsubsection*{$L^2$-dimension of homology and cohomology}



Recall that $\Ext_{\N G}^i(-,\N G)$ and
$\Tor^{\N G}_i(\N G,-)$ are (left) $\N G$-modules viewing $\N G$ as a
bimodule.

\begin{lemma}
  \label{lemma:l2-dim-ext-tor-with-NG}
  For any finitely generated module $A$ over $\N G$, $\ldim \Ext^1_{\N
    G}(A,\N G) = 0 = \ldim \Tor^{\N G}_1(\N G, A)$.
\end{lemma}

\begin{proof}
  Choose $0\to M \to P \to A \to 0$ with $P$ finitely generated
  projective.  Since
  \[
  0\to A^* \to P^* \to M^* \to \Ext_{\N G}^1 (A, \N G) \to 0
  \]
  we have, by Lemma~\ref{lemma:l2-dim-dual}, 
  \begin{align*}
    \ldim \Ext_{\N G}^1 (A, \N G) & = \ldim M^* - \ldim P^* + \ldim
    A^* \\
    & = \ldim M - \ldim P + \ldim A = 0.
  \end{align*}
%
%
%
%
%
%
%
  Similar argument works for $\Tor_1^{\N G}$.
\end{proof}

Following \cite{Cha-Orr:2009-01}, we say that a homomorphism $f\colon
A\to B$ between $\N G$-modules is an \emph{$L^2$-equivalence} if
$\ldim \Ker f = 0 = \ldim \Coker f$.

\begin{proposition}
  \label{proposition:l2-dim-homology-cohomology}
  Suppose $C_*$ is a chain complex with each $C_i$ finitely generated
  and projective over~$\N G$.  Then for any $n$, the Kronecker evaluation map
  \[
  H^n(\Hom(C_*,\N G)) \to \Hom(H_n(C_*),\N G)
  \]
  is an $L^2$-equivalence.  Consequently, $\ldim H_n(C_*) = \ldim
  H^n(\Hom(C_*,\N G))$.
\end{proposition}

\begin{proof}
  By \cite[p.~239]{Lueck:2002-1}, $\N G$ is semihereditary, namely,
  any finitely generated submodule of a projective module is
  projective.  It follows that $H_i(C_*)$ has projective dimension at
  most 1, i.e., $\Ext^p_{\N G}(H_i(C_*),M)=0$ for any $M$ and $p\ge
  2$.  (For, it is seen easily that both the boundary and cycle
  submodules in $C_i$ are finitely generated and projective.)

  Consider the universal coefficient spectral sequence
  \[
  E^{p,q}_2 = \Ext^p(H_q(C_*), \N G) \Longrightarrow
  H^n(\Hom(C_*,\N G)).
  \]
  By the above observation and
  Lemma~\ref{lemma:l2-dim-ext-tor-with-NG}, all the $E_2$-terms have
  $L^2$-dimension zero except (possibly) the bottom row
  $E^{0,q}_2=\Hom(H_q(C_*),\N G)$.  That is, the spectral sequence
  ``$L^2$-collapses'' at the $E_2$ terms.  It follows that the
  Kronecker evaluation map is an $L^2$-equivalence.  By
  Lemma~\ref{lemma:l2-dim-dual}, the conclusion on the $L^2$-dimension
  follows.
\end{proof}

\subsubsection*{$L^2$-signatures of hermitian forms over $\N G$}
\label{subsubsection:L2-metabolizer}


Here we discuss a formulation of $L^2$-signatures using only $\N G$ as
a ring, without making any explicit use of the Hilbert space $\lt G$
(particularly the traditional $L^2$-homology $H_*^{(2)}(-;\lt G)$).
While this setup seems to be known, the author could not find a
treatment in the literature.  Since it is the right setup for out
later arguments, we give a quick outline here.  (An elementary through
treatment with full details will be given in a separate introductory
article of the author.)

We call an $\N G$-module homomorphism $\lambda\colon A \to A^*$ with
$A$ finite generated a \emph{sesquilinear form} over $\N G$, viewing
it often as $\lambda \colon A\times A \to \N G$.  We say that
$\lambda$ is \emph{hermitian} if $A$ is finitely generated and
$\lambda(x)(y)=(\lambda(y)(x))^*$ for $x,y\in A$.  We say that
$\lambda$ is \emph{$L^2$-nonsingular} if $\lambda$ is an
$L^2$-equivalence as a homomorphism $A\to A^*$.

Suppose $\lambda\colon A \to A^*$ is a hermitian form over~$\N G$.
While $A$ is not necessarily projective, $T(A)$ always lies in the
kernel of $\lambda$ by the functoriality of $T(-)$ and by
Lemma~\ref{lemma:l2-dim-dual}, and thus $\lambda$ gives rise to a
hermitian form $P(\lambda)$ on the projective module~$P(A)$.  Taking
$Q$ satisfying $P(A)\oplus Q=(\N G)^n$ ($n<\infty$), since
$P(\lambda)\oplus 0$ on $P(A)\oplus Q$ is a hermitian form over the
free module $(\N G)^n$, we can apply functional calculus to obtain an
orthogonal decomposition
\[
(\N G)^n= A_+ \oplus A_- \oplus A'_0
\]
such that $P(\lambda)\oplus 0$ is positive definite, negative
definite, and zero on the submodules $A_+$, $A_-$, and $A'_0$,
respectively.  Here we say that $\lambda$ is positive (resp.\
negative) definite on $A$ if for any nonzero $x\in A$, $\lambda(x,x)$
is of the form $a^*a$ (resp.\ $-a^*a$) for some nonzero $a \in \N G$.
In fact, since $Q\subset A_0'$ is a summand, for $A_0=A'_0\cap P(A)$,
we also have a decomposition
\[
P(A)=A_+\oplus A_-\oplus A_0.
\]


Now the
\emph{$L^2$-signature} of $\lambda$ is defined by
\[
\lsign \lambda = \ldim A_+ - \ldim A_- \in \R.
\]
$\lsign$ is well-defined, e.g., by carrying out a standard proof for
Sylvester's law of inertia, using $\ldim$ in place of the ordinary
complex dimension.  (A similar argument is given in the proof of
Proposition~\ref{proposition:l2-metabolizer} below.)

It is known that $\lsign$ induces a real-valued homomorphism of the
$L$-group $L^0(\N G)$.  We need a slightly stronger fact, namely,
the following $L^2$-analogue of ``topologist's signature vanishing
criterion'':

\begin{proposition}
  \label{proposition:l2-metabolizer}
  If $\lambda\colon A\to A^*$ is an $L^2$-nonsingular hermitian form
  over $\N G$ and there is a submodule $H\subset A$ satisfying
  $\lambda(H)(H)=0$ and $\ldim H \ge \frac12 \ldim A$, then $\lsign
  \lambda = 0$ and $\ldim H=\frac12 \ldim A$.
\end{proposition}

\begin{proof}
  A standard argument for the case of finite dimension complex vector
  spaces works, using $\ldim$ in place of the complex dimension.  For
  concreteness we describe details below:
  let $r=\frac12 \ldim A$.  Choose an orthogonal decomposition $A=A_+
  \oplus A_- \oplus A_0$ such that $\lambda$ is positive (resp.\
  negative) definite on $A_+$ (resp.\ $A_-$) and $A_0=\Ker \lambda$.
  Suppose $\ldim A_+> r$.  Since $\ldim H \ge r$, $\ldim H\cap A_+>0$.
  But since $\lambda$ is zero on $H$ and positive definite on $A_+$,
  $H\cap A_+=0$.  From this contradiction it follows that $\ldim A_+
  \le r$, and similarly $\ldim A_- \le r$.  By the
  $L^2$-nonsingularity, $\ldim A_0=0$ and $\ldim A_++\ldim A_- = 2r$.
  Therefore $\ldim A_+ = r = \ldim A_-$.  If $\ldim H > r$, then
  $\ldim H\cap A_+>0$, which is a contradiction.
%
\end{proof}

\subsection{More about $L^2$-homology and amenable groups in Strebel's
  class}
\label{subsection:more-on-L2-homology-and-D(R)}

In this subsection we generalize
Theorem~\ref{theorem:cha-orr-local-property-NG}.  We begin with an
observation that $\N G$ looks like a flat $\C G$-module to
the eyes of $L^2$-dimension, based on L\"uck's result on the vanishing
of the $L^2$-dimension of~$\Tor$.  Here we use the amenability
crucially.

\begin{lemma}
  \label{lemma:NG-tensoring-ker-im-coker}
  Suppose $G$ is amenable.  Then for a $\C G$-homomorphism $f\colon
  A\to B$, the natural $\N G$-module homomorphisms $\N G\otimes \Ker f
  \to \Ker\{1_{\N G}\otimes f\}$ is an $L^2$-equivalence.  The
  analogues for $\Im\{-\}$ and $\Coker\{-\}$ hold as well.
\end{lemma}

Note that the cokernel part is obvious since tensoring is right exact.

\begin{proof}
  Tensoring with $\N G$ over $\C G$ gives us the following commutative
  diagram with exact row and column:
  \[
  \xymatrix @R=1.3em@C=1.2em{
    & \kern-2em\Tor^{\C G}_1(\N G, \Im f)\kern-2em \ar[d] \\
    & \N G \otimes \Ker f \ar[d]\\
    & \N G \otimes A \ar[d]\ar[dr]^{1_{\N G} \otimes f}\\
    \Tor^{\C G}_1(\N G, \Coker f) \ar[r]&
    \N G \otimes \Im f \ar[r]\ar[d] &
    \N G \otimes B \ar[r] & 
    \N G \otimes \Coker f \ar[r] &
    0 \\
    & 0
  }
  \]
  Due to L\"uck \cite[p.~259]{Lueck:2002-1}, $\ldim \Tor_1^{\C G}(\N
  G, -)=0$ since $G$ is amenable.  The conclusions follow by
  straightforward arguments.
\end{proof}

Now we introduce a lemma which makes crucial use of Strebel's class.
Let $R$ be a commutative ring with unity.  Recall that, a group $G$ is
said to be in $D(R)$ if a homomorphism $f\colon M\to N$ between
projective $RG$-modules is injective whenever $1_R\otimes f\colon
R\otimes M \to R\otimes N$ is injective~\cite{Strebel:1974-1}.
Reformulating this definition slightly, we have: a group $G$ is in
$D(R)$ if and only if for any homomorphism $f\colon M\to N$ between
$RG$-modules with $N$ projective, the restriction of $f$ on a
projective submodule $P \subset M$ is injective whenever the induced
map $R\otimes P \to R\otimes N$ is injective.

Obviously a homomorphism of an $R$-free module, say $R^I$, into
$R\otimes M$ lifts to a homomorphism of the $RG$-free module $(RG)^I
\to M$.  Thus, we have the following consequence:

\begin{lemma}
  \label{lemma:lemma:rank-of-image-D(R)}
  Suppose $G$ is in~$D(R)$.  Then for any homomorphism $f\colon
  M\to N$ between $RG$-modules with $N$ projective, the maximal
  rank of $RG$-free submodules in $\Im f$ is greater than or
  equal to the maximal rank of $R$-free submodules in $\Im\{1_R\otimes
  f\colon R\otimes M \to R\otimes N\}$.
\end{lemma}

Combining the above results obtained from amenability and Strebel's
condition, we obtain the following $L^2$-dimension estimates.  (cf.\
\cite[Lemma~4.4]{Cochran-Harvey:2004-1} for algebraic analogues for
PTFA groups.)

\begin{lemma}
  \label{lemma:L2-dim-of-im-and-ker}
  Suppose $G$ is amenable and in $D(R)$, $R=\Q$ or $\Z_p$, and
  $\phi\colon M\to N$ is a homomorphism between $\Z G$-modules.
  Denote the induced maps on $R\otimesover{\Z G}-$ and
  $\N G\otimesover{\Z G}-$ by $1_R\otimes\phi$ and
  $1_{\N G}\otimes\phi$.
  \begin{enumerate}
  \item If $N$ is projective, then $\dim_R\Im \{1_R\otimes\phi\} \le
    \ldim_G \Im \{1_{\N G}\otimes\phi\}$.

  \item If, in addition, $M$ is finitely generated and projective,
    then $\dim_R\Ker \{1_R\otimes\phi\} \ge \ldim_G \Ker \{1_{\N
      G}\otimes\phi\}$.
  \end{enumerate}
\end{lemma}

\begin{proof}
  (1) Let $d=\dim_R \Im \{1_R\otimes \phi\}$.  ($d$ may be any
  cardinal.)  Applying Lemma~\ref{lemma:lemma:rank-of-image-D(R)} to
  $f=1_{RG}\otimes\phi \colon RG\otimes M \to RG\otimes N$, we obtain
  an injection $(RG)^d \to \Im\{1_{RG}\otimes\phi\} \subset RG\otimes
  N$.  We may assume that it is induced by a homomorphism $i\colon (\Z
  G)^d \to \Im \phi \subset N$ (by multiplying it by an integer if
  necessary).  Note that for any two index sets, say $I$ and $J$, of
  arbitrary cardinality, a homomorphism $\Z^I \to \Z^J$ is injective
  whenever the induced map $R^I \to R^J$ is injective.  Especially, in
  our case, $i$ is injective.  Now,
  \[
  (\N G)^d = \N G \otimesover{\Z G}(\Z G)^d \cong \N G \otimesover{\Z
    G}(\Im i) \to \Im\{ 1_{\N G} \otimes i\}
  \]
  is an $L^2$-equivalence, by
  Lemma~\ref{lemma:NG-tensoring-ker-im-coker}.  Since $\Im\{1_{\N
    G}\otimes i \} \subset \Im\{ 1_{\N G} \otimes \phi\}$, we have
  \[
  d = \ldim \Im\{1_{\N G} \otimes i\} \le \ldim \Im\{ 1_{\N G} \otimes
  \phi\}.
  \]
  (2) We may assume that $M$ is free by taking a projective module $P$
  such that $M\oplus P$ is finitely generated and free, and replacing
  $\phi$ with $\phi\oplus 1_P\colon M\oplus P \to N\oplus P$.  Now,
  combine (1) with the equality
  \begin{multline*}
    \ldim \Ker\{1_{\N G}\otimes\phi \} + \ldim \Im\{ 1_{\N
      G}\otimes\phi\} = \ldim \N G \otimes M
    \\
    = \dim_R R\otimes M = \dim_R \Ker\{1_{R}\otimes\phi\} + \dim_R
    \Im\{1_{R}\otimes\phi\}. \qedhere
  \end{multline*}
\end{proof}

Using the above results, we give bounds for the $L^2$-dimension of $\N
G$-coefficient homology in terms of the ordinary dimension of
$R$-coefficient homology.

\begin{theorem}
  \label{theorem:L2-dim-estimate-of-NG-homology}
  \leavevmode\Nopagebreak
  \begin{enumerate}
  \item Suppose $G$ is amenable and in $D(R)$ with $R=\Q$ or $\Z_p$,
    and $C_*$ is a projective chain complex over $\Z G$ with $C_n$
    finitely generated.  Then we have
    \[
    \ldim H_n(\N G \otimesover{\Z G}C_*) \le \dim_R
    H_n(R\otimesover{\Z G}C_*).
    \]
  \item In addition, if $\{x_i\}_{i\in I}$ is a collection of
    $n$-cycles in $C_n$, then for the submodules $H\subset H_n(\N G
    \otimes C_*)$ and $\overline H \subset H_n(R\otimes C_*)$
    generated by $\{[1_{\N G}\otimes x_i]\}_{i\in I}$ and
    $\{[1_{R} \otimes x_i]\}_{i\in I}$, respectively, we have
    \[
    \ldim H_n(\N G \otimes C_*)-\ldim H \le \dim_R H_n(R\otimes C_*) -
    \dim_R \overline{H}.
    \]
  \end{enumerate}
\end{theorem}

\begin{proof}
  First we prove (2).  Let $\partial_n \colon C_n \to C_{n-1}$ be the
  boundary map, and define $f\colon (\Z G)^I\oplus C_{n+1} \to C_n$
  by $f(e_i,v) \mapsto x_i+\partial_{n+1}(v)$ where $\{e_i\}_{i\in I}$
  is the standard basis of $(\Z G)^I$.  Then we have
  \begin{gather*}
    H_n(\N G \otimes C_*)/H =
    \Ker\{1_{\N G}\otimes \partial_n\} / \Im \{1_{\N G}\otimes f\}\\
    H_n(R\otimes C_*)/\overline{H} = \Ker\{1_R \otimes \partial_n\} /
    \Im \{1_{R}\otimes f\}
  \end{gather*}
  Applying Lemma~\ref{lemma:L2-dim-of-im-and-ker}, we have
  \begin{multline*}
    \ldim H_n(\N G\otimes C_*)/H = \ldim
    \Ker\{1_{\N G}\otimes \partial_n\} - \ldim
    \Im\{1_{\N G}\otimes f\} \\
    \le \dim_R \Ker\{1_R \otimes \partial_n\} - \dim_R
    \Im\{1_R\otimes f\} = \dim_R H_n(R\otimes C_*)/\overline{H}.
  \end{multline*}
  This proves (2).  Considering the special case of $H=0=\overline H$,
  (1) follows.
\end{proof}

Now we apply the above result on the $L^2$-dimension to $L^2$-homology
of spaces.
For a space $X$ and a field $R$, we denote the $i$-th Betti number by
$b_i(X;R)=\dim_R H_i(X;R)$.  For $R=\Q$, we simply write
$b_i(X)=b_i(X;\Q)$.
When $X$ is endowed with $\pi_1(X)\to G$, we define the
\emph{$L^2$-Betti number} of $X$ by $b^{(2)}_i(X;\N G)=\ldim H_i(X;\N
G)$.  When the choice of $G$ is obvious, we write $b^{(2)}_i(X) =
b^{(2)}_i(X;\N G)$.

We remark that for a finite CW complex $X$, it easily follows, from
the standard properties of the $L^2$-dimension we mentioned above,
that the $L^2$ Euler characteristic $\sum_i (-1)^i \cdot b^{(2)}_i(X)$
is equal to the ordinary Euler characteristic~$\chi(X)$.

\begin{lemma}
  \label{lemma:vanishing-of-l2-H1}
  Suppose $X$ is a connected finite CW complex and $\phi\colon
  \pi_1(X) \to G$ is a homomorphism with $G$ amenable and in $D(R)$.
  If there is a loop $\alpha$ in $X$ such that $\phi([\alpha]) \in G$
  has infinite order and the homology class of $\alpha$ generates
  $H_1(X;R)$, then $b^{(2)}_0(X)=0 = b^{(2)}_1(X)$.
\end{lemma}

\begin{proof}
  Choose $S^1\to X$ representing~$\alpha$.  Then $H_i(X,S^1;R)=0$ for
  $i=0,1$.  By Theorem~\ref{theorem:cha-orr-local-property-NG} (or
  Theorem~\ref{theorem:L2-dim-estimate-of-NG-homology}), it follows
  that $\ldim H_1(X,S^1;\N G)=0$.  Therefore $b^{(2)}_i(X) \le
  b^{(2)}_i(S^1)$ for $i=0,1$.  Since $\phi$ induces an injection
  $\pi_1(S^1)\to G$, the $G$-cover of $S^1$ is a disjoint union of
  lines.  Therefore $H_i(S^1;\N G)=0$ and $b^{(2)}_i(S^1)=0$ for all
  $i\ge 1$.  Since $\chi(S^1)=0$, $b^{(2)}_0(S^1)=0$.
\end{proof}

\begin{lemma}
  \label{lemma:dim-of-l2-H2}
  Suppose $(X,\partial X)$ is a 4-dimensional Poincar\'e duality pair
  (e.g., a compact 4-manifold with boundary) with $X$ and $\partial
  X$ connected, which is endowed with $\pi_1(X) \to G$ where
  $G$ is amenable and in $D(R)$.  Suppose $b^{(2)}_0(\partial
  X)$, $b^{(2)}_0(X)$, and $b^{(2)}_1(X)$ vanish, $b_1(X;R)=1$, and
  the inclusion-induced map $H_1(\partial X;R) \to H_1(X;R)$ is
  nontrivial.  Then $b^{(2)}_2(X)=b_2(X;R)$.
\end{lemma}

We remark that if there is a loop $\alpha$ in $\partial X$ satisfies
the hypotheses of Lemma~\ref{lemma:vanishing-of-l2-H1} for both $X$
and $\partial X$, then Lemma~\ref{lemma:dim-of-l2-H2} applies.  For
example, it applies to an $(h)$-solution for a knot.

\begin{proof}
  From the hypothesis and the long exact sequence for $(X,\partial
  X)$, we have $b^{(2)}_1(X,\partial X)=0$.  By duality and
  Proposition~\ref{proposition:l2-dim-homology-cohomology}, we have
  $b^{(2)}_3(X)=b^{(2)}_1(X,\partial X)=0$, and $b^{(2)}_i(X)=0$ for
  $i\ge 4$.  Similarly
  $b_3(X;R)=b_1(X,\partial X)=0$, $b_i(X;R)=0$ for $i\ge 4$.  Since
  the $L^2$ Euler characteristic is equal to the ordinary Euler
  characteristic, it follows that $b^{(2)}_2(X) = b_2(X;R)$.
\end{proof}

\subsection{Proof of
  Theorem~\ref{theorem:n.5-solvable-obstruction-amenable-in-D(R)}}
\label{subsection:proof-n.5-solvable-obstruction}

We start with an observation on the integral and modulo $p$ Betti
numbers:

\begin{lemma}
  \label{lemma:mod-p-betti-number-of-solution}
  If $W$ is a 4-manifold such that both $H_1(W)$ and
  $\Coker\{H_1(\partial W)\to H_1(W)\}$ has no $p$-torsion, then
  $b_i(W)=b_i(W;\Z_p)$ for any~$i$.
\end{lemma}

As a special case, if $W$ is an integral $(h)$-solution of a knot,
then the conclusion of
Lemma~\ref{lemma:mod-p-betti-number-of-solution} holds.

\begin{proof}
  The conclusion is obvious for $i=1$.  Let $H=\Coker\{H_1(\partial
  W)\to H_1(W)\}$ and $r=\rank_\Z \Ker\{H_0(\partial W)\to H_0(W)\}$.
  Then $b_3(W)=b_1(W,\partial W)=r+\rank_\Z H$ and
  $b_3(W;\Z_p)=r+\rank_{\Z_p} (H\otimes_\Z \Z_p)$.  From the
  hypothesis, the conclusion for $i=3$ follows.  By the universal
  coefficient theorem $H_3(W;\Z_p)=(H_3(W)\otimes \Z_p) \oplus
  \Tor_1^\Z(H_{2}(W),\Z_p)$, it follows that both $H_3(W)$ and
  $H_2(W)$ have no $p$-torsion.  By the universal coefficient theorem
  again, the conclusion for $i=4$ and $i=2$ follows.
\end{proof}

\begin{proof}
  [Proof of
  Theorem~\ref{theorem:n.5-solvable-obstruction-amenable-in-D(R)}]
  First we consider the case of $R=\Z_p$.  Suppose $K$ is a knot with
  meridian $\mu$ and zero-surgery manifold $M(K)$, and $W$ is an
  integral $(n.5)$-solution for~$M(K)$.  Let $\{\tilde
  x_1,\ldots,\tilde x_r\}$ and $\{y_1,\ldots,y_r\}$ be an
  $(n+1)$-lagrangian and $(n)$-dual as in
  Definition~\ref{definition:solvable-knots} where $2r=\rank_\Z
  H_2(W)$.  Suppose $\phi\colon \pi_1(X) \to G$ is a homomorphism, $G$
  is amenable and in $D(\Z_p)$, $G^{(n+1)}=\{e\}$, and $\phi([\mu])$
  has infinite order in~$G$.

  From the meridian condition above, it follows that
  $b^{(2)}_0(M(K))=0=b^{(2)}_1(M(K))$ by
  Lemma~\ref{lemma:vanishing-of-l2-H1}.  Also, since $H_1(M(K))\cong
  H_1(W)$, $b^{(2)}_2(W)=b_2(W;\Z_p)=b_2(W)=2r$ by
  Lemma~\ref{lemma:dim-of-l2-H2} and
  Lemma~\ref{lemma:mod-p-betti-number-of-solution}.

  We will compute the von Neumann $\rho$-invariant of $(M,\phi)$
  using~$W$.  First, since the images of the $\tilde x_i$ and $y_i$
  form a lagrangian and its dual in $H_1(W;\Q)$, it follows that the
  ordinary signature of $W$ vanish.  To compute the $L^2$-signature of
  $W$, consider the intersection form
  \[
  \lambda\colon H_2(W;\N G) \times H_2(W;\N G)\to \N G.
  \]
  Since $b^{(2)}_2(M(K))=0$ and $b^{(2)}_1(M(K))=0$, $H_2(W;\N G) \to
  H_2(W,M(K);\N G)$ is an $L^2$-equivalence.  Therefore $\lambda$ is
  $L^2$-nonsingular.  Since $G^{(n+1)}=\{e\}$, $\pi_1(W)\to G$ factors
  through $\pi_1(W)/\pi_1(W)^{(n+1)}$ and there is an induced map
  \[
  H_2(W;\Z[\pi_1(W)/\pi_1(W)^{(n+1)}]) \to H_2(W;\N G).
  \]
  Let $H\subset H_2(W; \N G)$ and $\overline H \subset H_2(W;\Z_p)$ be
  the submodules generated by the images of the $\tilde x_i$.  By the
  existence of the $(n)$-dual, the images of the $\tilde x_i$ in
  $H_2(W;\Z_p)$ are linearly independent over the field $\Z_p$.
  Therefore, $\overline H$ has dimension $r$ over $\Z_p$.  From this,
  by Theorem~\ref{theorem:L2-dim-estimate-of-NG-homology}, it follows
  that
  \[
  \ldim H \ge b^{(2)}_2(W) - b_2(W;\Z_p) + \dim_{\Z_p} \overline H = r.
  \]
  By Proposition~\ref{proposition:l2-metabolizer}, it follows that
  $\lsign_G(W)=0$.

  The same argument also applies to $R=\Q$; in this case we do not
  need to appeal to Lemma~\ref{lemma:mod-p-betti-number-of-solution}
  to deal with the $\Z_p$-coefficient Betti number, since we can use
  $\Q$ when we apply
  Theorem~\ref{theorem:L2-dim-estimate-of-NG-homology}.
\end{proof}

\section{Examples and computation}
\label{section:examples-and-computation}

In this section, we give examples of knots which illustrate the
usefulness of the new obstructions.  The knots have very subtle
behavior in the knot concordance group---these are $(n)$-solvable, and
not distinguished from $(n.5)$-solvable knots by any PTFA
$L^2$-signature obstructions, but not $(n.5)$-solvable.  This is
detected using
Theorem~\ref{theorem:n.5-solvable-obstruction-amenable-in-D(R)}.

In order to construct non-PTFA amenable coefficient systems lying in
Strebel's class to which we apply
Theorem~\ref{theorem:n.5-solvable-obstruction-amenable-in-D(R)}, we
need some refined ingredients: mixed-coefficient commutator series and modulo
$p$ noncommutative Blanchfield form.  We begin by introducing the
former, and then proceed to the construction and analysis of the
examples.

\subsection{Mixed-coefficient commutator series}
\label{subsection:mixed-coefficient-commutator-series}

Suppose $G$ is a group and $\cP = (R_0,R_1,\ldots)$ is a sequence of
commutative rings with unity.

\begin{definition}
  \label{definition:mixed-coefficient-commutator-series}
  The \emph{$\cP$-mixed-coefficient commutator series} $\{\cP^k G\}$
  of $G$ is defined by $\cP^0G =G$,
  \[
  \let\to=\rightarrow
  \cP^{k+1}G = \Ker\Big\{ \cP^k G \to \frac{\cP^k G}{[\cP^k G,\cP^k
    G]} \to \frac{\cP^k G}{[\cP^k G,\cP^k
    G]}\otimesover{\Z} R_k\Big\}.
  \]
\end{definition}

We remark that ${\cP^k G}/{[\cP^k G,\cP^k G]}$ and $({\cP^k G}/{[\cP^k
  G,\cP^k G]})\otimes R_k$ can be identified with
$H_1(G;\Z[{G}/{\cP^kG}])$ and
$H_1(G;R_{k}[{G}/{\cP^kG}])$, respectively.  It can be
verified that $\cP^k G$ is a characteristic normal subgroup of $G$ for
any $\cP$ and~$k$.

\begin{example}
  \leavevmode\Nopagebreak
  \begin{enumerate}
  \item If $\cP=(\Z,\Z,\ldots)$, then $\{\cP^k G\}$ is the
    standard derived series~$\{G^{(k)}\}$.
  \item If $\cP=(\Q,\Q,\ldots)$, then $\{\cP^k G\}$ is the rational
    derived series of~$G$.
  \end{enumerate}    
\end{example}


\begin{lemma}
  \label{lemma:mixed-coefficient-commutator-series-is-amenable-in-D(R)}
  Suppose $\cP = (R_0,R_1,\ldots)$, $p$ is a fixed prime, and $R_k$
  has the property that any integer relatively prime to $p$ is
  invertible in $R_k$ for $k< n$.  Then for any group $G$ and $k\le
  n$, $G/\cP^k G$ is amenable and lies in~$D(\Z_p)$.
\end{lemma}

\begin{proof}
  $H_1(G;R_k[G/\cP^kG])$ is a module over $R_k$ and consequently it
  has no torsion coprime to~$p$ as an abelian group.  Since the $\cP^k
  G/ \cP^{k+1} G$ is a subgroup of $H_1(G;R_k[G/\cP^kG])$, the
  conclusion follows from Lemma~\ref{lemma:poly-mixed-groups}.
\end{proof}

For example, if $R_k$ is either $\Z_p$ or $\Q$ for each $k$, then
$G/\cP^n G$ is amenable and lies in $D(\Z_p)$ by
Lemma~\ref{lemma:mixed-coefficient-commutator-series-is-amenable-in-D(R)}.

\subsection{Knots with vanishing PTFA $L^2$-signature obstructions}
\label{subsection:knots-with-vanishing-COT-obstruction}

We begin by recalling a standard satellite construction, which has
been used in several papers on knot concordance.  It is often called
``infection'' or ``generalized doubling''.

For a knot or link $K$ in $S^3$, we denote its exterior by
$X(K)=S^3-\inte N(K)$ where $N(-)$ is a tubular neighborhood.  Let $K$
be a knot in $S^3$, and $\eta$ be a simple closed curve in $X(K)$
which is unknotted in~$S^3$.  Choose $N(\eta)$ disjoint from $K$,
remove $\inte N(\eta)$ from $S^3$, and then fill in it with the
exterior $X(J)$ of another knot $J$ along an orientation reversing
homeomorphism on the boundary which identifies the meridian (resp.\
0-linking longitude) of $J$ with the 0-linking longitude (resp.\
meridian) of~$\eta$.  The result is homeomorphic to $S^3$ again, but
$K$ becomes a new knot in the resulting $S^3$, which we denote
by~$K(\eta,J)$.

It is easily seen that if $J$ and $J'$ are concordant, then
$K(\eta,J)$ and $K(\eta,J')$ are concordant.




\begin{proposition}
  \label{proposition:l2sign-satellite-n-solution}
  Suppose $K$ is slice, $[\eta]\in (\pi_1 X(K))^{(n)}$, and
  $\Arf(J)=0$.  Then there exists an $(n)$-solution $W$ for
  $K(\eta,J)$ satisfying the following: for any homomorphism
  $\phi\colon \pi_1 M(K(\eta,J)) \to G$ extending to $W$, if $G$ is
  amenable and in $D(R)$ for some $R$, then we have
  $\rhot(M(K),\phi)=\rhot(M(J),\psi)$ where $\psi\colon \pi_1 M(J) \to
  \Z_d$ is the surjection sending a meridian to $1\in \Z_d$ and $d$ is
  the order of $\phi([\eta])$ in~$G$.  (If $d=\infty$, $\Z_d$ is
  understood as~$\Z$.)
\end{proposition}

We remark that in
Proposition~\ref{proposition:l2sign-satellite-n-solution} we do not
assume $G^{(n+1)}=\{e\}$.  Before proving
Proposition~\ref{proposition:l2sign-satellite-n-solution}, we give a
consequence, which produces a large family of knots for which PTFA
$L^2$-signature obstructions vanish.  For this purpose we need the
following formula for abelian $\rho$-invariants of knots.  We denote
the Levine-Tristram signature function of a knot $J$ by
$\sigma_J\colon S^1\to \Z$.  Namely, choosing a Seifert matrix $A$ of
$J$, it is defined by $\sigma_J(w)=\sign\big( (1-w)A + (1-\overline
w)A^T \big)$.

\begin{lemma}[{\cite[Proposition 5.1]{Cochran-Orr-Teichner:2002-1},
    \cite[Corollary~4.3]{Friedl:2003-5},
    \cite[Lemma~8.7]{Cha-Orr:2009-01}}]
  \label{lemma:abelian-rho-invariant}
  Suppose $J$ is a knot with meridian $\mu$ and $\alpha\colon
  \pi_1M(J)\to G$ is a homomorphism whose image contained in an
  abelian subgroup of~$G$.  Then
  \[
  \rhot(M(J),\alpha) = \begin{cases} \int_{S^1} \sigma_{J}(w)\,dw &
    \text{if $\alpha([\mu])\in G$ has
      infinite order,} \\[1ex]
    \sum_{r=0}^{d-1} \sigma_{J}(e^{2\pi r\sqrt{-1}/d}) &\text{if
      $\alpha([\mu])\in G$ has order $d<\infty$.}
  \end{cases}
  \]
\end{lemma}

\begin{corollary}
  \label{corollary:vanishing-of-COT-obstruction}
  Suppose $K$ is slice, $[\eta]\in \pi_1 X(K)^{(n)}$, $\Arf(J)=0$, and
  $\int_{S_1} \sigma_{J}(w)\,dw=0$.  Then there is an $(n)$-solution
  $W$ for $J'=K(\eta,J)$ satisfying the following: for any
  homomorphism $\phi\colon \pi_1 M(J') \to G$ into a torsion-free
  amenable group $G$ lying in $D(R)$ that extends to $W$, we have
  $\rhot(M(J'),\phi)=0$.
\end{corollary}

\begin{proof}
  Observing that the order of any element in a torsion-free group $G$
  is either $0$ or~$\infty$, the conclusion follows immediately from
  Proposition~\ref{proposition:l2sign-satellite-n-solution} and
  Lemma~\ref{lemma:abelian-rho-invariant}.
\end{proof}

Since PTFA groups are torsion-free amenable and in $D(R)$, the
conclusion of Corollary~\ref{corollary:vanishing-of-COT-obstruction}
says that the PTFA $L^2$-signature obstruction due to
Cochran-Orr-Teichner \cite[Theorem~4.2]{Cochran-Orr-Teichner:1999-1}
does not distinguish our $J'=K(\eta,J)$ from $(n.5)$-solvable knots,
and consequently any prior methods (e.g.,
\cite{Cochran-Teichner:2003-1, Cochran-Kim:2004-1,
  Cochran-Harvey-Leidy:2008-2, Cochran-Harvey-Leidy:2009-1,
  Cochran-Harvey-Leidy:2009-02, Cochran-Harvey-Leidy:2009-03}) do not
neither, since all these essentially depend on
\cite[Theorem~4.2]{Cochran-Orr-Teichner:1999-1}.  The following
definition formalizes the property:

\begin{definition}
  \label{definition:knots-with-vanishing-PTFA-signature-obstructions} 
    If $K$ admits an $(n)$-solution $W$ such that $\rhot(M(K),\phi)=0$
    for any $\phi\colon \pi_1 M(K) \to G$ into a PTFA group $G$ which
    extends to $W$, then we say that \emph{$K$ is an $(n)$-solvable
      knot with vanishing PTFA $L^2$-signature obstructions}.
\end{definition}

Obviously the knot $J'$ obtained by the satellite construction in
Corollary~\ref{corollary:vanishing-of-COT-obstruction} is an
$(n)$-solvable knot with vanishing PTFA $L^2$-signature obstructions.

Let $\mathcal{V}_n$ be the set of concordance classes of
$(n)$-solvable knots with vanishing PTFA $L^2$-signature obstructions.

\begin{proposition}
  \label{proposition:vanishing-PTFA-subgroup}
  For any $n$, $\mathcal{V}_n$ is a subgroup of the $n$th term $\F_n$
  of the solvable filtration.
\end{proposition}

\begin{proof}
  Suppose $W_1$ and $W_2$ are $(n)$-solutions of two knots $K_1$ and
  $K_2$ satisfying
  Definition~\ref{definition:knots-with-vanishing-PTFA-signature-obstructions},
  respectively.  Let $K=K_1\mathbin{\#}K_2$.  It is known that there
  is a standard cobordism $C$ between $M(K)$ and $M(K_1)\cup M(K_2)$
  such that $\lsign_G(C)=0=\sign(C)$ for any $\pi_1(C)\to G$
  (e.g. see~\cite[p.~113]{Cochran-Orr-Teichner:2002-1},
  \cite[p.~1429]{Cochran-Harvey-Leidy:2009-1}).  Glueing $C$ with
  $W_1\cup W_2$ along $M(K_1)\cup M(K_2)$, we obtain an $(n)$-solution
  $W$ for~$K$.  If $\psi\colon \pi_1(W)\to G$ is a homomorphism into a
  PTFA group $G$, then for its restrictions $\phi\colon \pi_1 M(K)\to
  G$ and $\phi_i\colon \pi_1 M(K_i) \to G$, we have
  $\rhot(M(K),\phi)=\rhot(M(K_1),\phi_1)+\rhot(M(M_2),\phi_2)$ since
  $C$ has vanishing ($L^2$-) signatures.  Each $\rhot(M(K_i),\phi_i)$
  vanishes by the hypothesis.  This shows $[K]\in\mathcal{V}_n$.
\end{proof}


\begin{proof}
  [Proof of Proposition~\ref{proposition:l2sign-satellite-n-solution}]

  In \cite[Lemma~5.4, Remark~5.7]{Cochran-Orr-Teichner:2002-1}, a
  $(0)$-solution $W_J$ for $J$ is constructed as follows, assuming
  $\Arf(J)$ vanishes.  (See also \cite[Theorem~5.10]{Cha:2003-1}.)
  Pushing a Seifert surface of $J$, we obtain a properly embedded
  framed surface $\Sigma$ in the 4-ball~$B^4$.  Let $g$ be the genus
  of~$\Sigma$.  Then the desired $W_J$ is obtained by removing a
  tubular neighborhood $N(\Sigma)$ from $B^4$ and attaching $(g$
  2-handles$)\times S^1$ along a collection of half basis curves
  on~$\Sigma$.

  We claim that $\pi_1(W_J)=\Z$, generated by a meridian.  To see
  this, first consider the manifold $N$ obtained by cutting $B^4$
  along the trace of pushing from the Seifert surface to~$\Sigma$.
  The group $\pi_1(B^4-\Sigma)$ is an HNN extension of~$\pi_1(N)$.
  Since $N$ is homeomorphic to $B^4$ itself, it follows that
  $\pi_1(B^4-\Sigma)=\Z$.  Since any curve on $\Sigma$ is
  null-homotopic in $N$, the attachment of $(g$ 2-handles$)\times S^1$
  has no effect on the fundamental group.

  In \cite[Proposition~3.1]{Cochran-Orr-Teichner:2002-1}, an
  $(n)$-solution $W$ with $\partial W=K(\eta,J)$ is constructed by
  glueing $W_J$ with a slice disk exterior $X$ of $K$, along a solid
  torus on the boundary.  It is known that
  $\rhot(M(K(\eta,J)),\phi)=\rhot(M(K),\psi)+\rhot(M(J),\psi_J)$,
  where $\psi\colon\pi_1(M(K))\to G$ and $\psi_J\colon \pi_1(M(J))\to
  G$ are restrictions of an extension $\pi_1(W)\to G$ of the
  given~$\phi$ (e.g., see
  \cite[Proposition~3.2]{Cochran-Orr-Teichner:2002-1},
  \cite[Lemma~2.3]{Cochran-Harvey-Leidy:2009-1}).  By
  Theorem~\ref{theorem:intro-obstruction-slice}, $\rhot(M(K),
  \psi)=0$.  Since $\psi_J$ factors through $\pi_1(W_J)=\Z$, the image
  of $\psi_J$ is exactly the cyclic subgroup in $G$ generated by the
  image of a meridian of~$J$.  By
  the induction property for $\rhot$, 
  it follows that $\rhot(M(J),\psi_J)$ is given as desired.
\end{proof}

Our main examples are obtained by iterating this satellite
construction as follows:

\medskip
\begin{list}{}
  {\def\makelabel#1{\emph{#1. }\hfil}%
    \settowidth\labelwidth{\makelabel{Output}}%
    \leftmargin=\parindent \advance\leftmargin by\labelwidth
    \labelsep=0mm\itemindent=0em
  }
\item[Input] A knot~$J_0$, a sequence of knots $K_0$, $K_1$, \ldots,
  $K_{n-1}$, and simple closed curves $\eta_k$ in $X(K_k)$ which are
  unknotted in~$S^3$. 
\item[Output] A sequence of knots $J_0$, $J_1,
  \ldots, J_n$ defined by $J_{k+1} = K_k(\eta_k, J_k)$.
\end{list}
\medskip

Observe that $X(J_n)=X(J_0) \cup X(K_{0}\cup\eta_{0}) \cup \cdots \cup
X(K_{n-1}\cup\eta_{n-1})$.  From this we obtain another way to
describe $J_n$: let $K_0'$ be the unknot, and let
$K_{k+1}'=K_k(\eta_k,K_k')$.  That is, $K_k'$ is the knot obtained by
using the unknot as $J_0$ in the above construction.  Then, regarding
$\eta_0 \subset X(K_0) \cup X(K_1\cup \eta_1) \cup \cdots \cup
X(K_{n-1}\cup\eta_{n-1}) = X(K_n')$, we can write
$J_n=K_n'(\eta_0,J_0)$.

\begin{lemma}
  \label{lemma:infection-curve-depth-for-iterated-satellite}
  If $\lk(K_k, \eta_k)=0$ for each $k$, then $[\eta_0] \in \pi_1
  X(K_n')^{(n)}$.
\end{lemma}

\begin{proof}
  By the assumption, $[\eta_k]$ is a product of conjugates of a
  meridian of $K_k$ in $\pi_1 X(K_k)$.  Since a meridian of $K_k$ is
  identified with a parallel of $\eta_{k+1}$ in the satellite
  construction, an induction shows that $[\eta_0] \in
  \pi_1(X(K_n'))^{(n)}$.
\end{proof}

From this observation, we immediately obtain the following fact by
applying Proposition~\ref{proposition:l2sign-satellite-n-solution} and
Corollary~\ref{corollary:vanishing-of-COT-obstruction}:

\begin{proposition}
  \label{proposition:solution-for-iterated-satellite}
  Suppose $K_k$ is slice and $\lk(K_k, \eta_k)=0$ for each $k$, and
  $\Arf(J_0)=0$.  Then there is an $(n)$-solution $W$ for $J_n$ such
  that if $\phi\colon \pi_1 M(J_n) \to G$ is a homomorphism extending
  to $W$ and $G$ is amenable and in $D(R)$ for some $R$, then
  $\rhot(M(J_n),\phi) = \rhot(M(J_0),\psi)$ where $\psi\colon \pi_1
  M(J_0)\to \Z_d$ is a surjection and $d$ is the order of an element
  in~$G^{(n)}$.  In addition, if $\int_{S_1} \sigma_{J_0}(w) \, dw =
  0$, then via the $(n)$-solution $W$, $J_n$ has vanishing PTFA
  $L^2$-signature obstruction.
\end{proposition}

\subsection{New non-solvable knots}
\label{subsection:non-solvability-of-examples}

We apply the iterated satellite construction starting with an infinite
family of initial knots $J^i_0$ ($i=1,2,\ldots$).  Fix $n$ and fix a
family of knots $\{K_k\mid k=0,\ldots, n-1\}$ together with simple
closed curves $\eta_k$ in $X(K_k)$, and let
$J^i_{k+1}=K_k(\eta_k,J_k^i)$.  (One may use the same $(K_k,\eta_k)$
for all~$k$.)



\begin{theorem}
  \label{theorem:linearly-independent-knots}
  Suppose $n>0$, $K_k$ is slice, $\lk(K_k, \eta_k)=0$, and the
  Alexander module $H_1(M(K_k);\Z[t^{\pm1}])$ is nontrivial and
  generated by $[\eta_k]$ for each~$k$.  Then there is an infinite
  collection $\{J^i_0 \mid i=1,2,\ldots\}$ for which the knots $J^i_n$
  $(i=1,2,\ldots)$ obtained as above are $(n)$-solvable and have
  vanishing PTFA $L^2$-signature obstructions, but any nontrivial
  linear combination $\mathop{\#}_i^{\vphantom{i}} a_i^{\vphantom{i}}
  J^i_n$ of the $J^i_n$ is not integrally $(n.5)$-solvable (and
  therefore not slice.)
\end{theorem}

Theorem~\ref{theorem:intro-linearly-independent-examples} in the
introduction immediately follows from
Theorem~\ref{theorem:linearly-independent-knots}.  The case of $n=0$
is classical, e.g., by Levine~\cite{Levine:1969-1} (and
Corollary~\ref{corollary:vanishing-of-COT-obstruction}).

In order to prove Theorem~\ref{theorem:linearly-independent-knots}, we
first describe how to choose~$\{J^i_0\}$.  Choose an infinite sequence
of $p_1, p_2,\ldots$ of distinct increasing primes which are greater
than the top coefficient of the Alexander polynomial
$\Delta_{K_0}(t)$.  Denote $w_i = e^{2\pi\sqrt{-1}/p_i}$.  For the
given $\{K_k\}$, let $L$ be the maximum of the Cheeger-Gromov bound
for the $\rho$-invariants of the $M(K_k)$ \cite{Cheeger-Gromov:1985-1}
(see also \cite{Cochran-Teichner:2003-1}).  Namely,
$|\rhot(M(K_k),\alpha)|< L$ for any $k=0,\ldots,n-1$ and for any
homomorphism $\alpha$ of~$\pi_1 M(K_k)$.

\begin{proposition}
  \label{proposition:infection-knot-construction}
  There is a family of knots $\{J^i_0\}$ satisfying the following:
  \begin{enumerate}
  \item[(C1)] $\sigma_{J^i_0}(w_i) = \sigma_{J^i_0}(w_i^{-1}) > nL$
    and $\sigma_{J^i_0}(w_i^r) = 0$ for $r\not\equiv \pm 1 \text{ mod
    } p_i$.
  \item[(C2)] $\sigma_{J^i_0}(w_j^r)=0$ for any $r$, whenever $i>j$.
  \item[(C3)] $\Arf(J^i_0)=0$ and $\int_{S^1} \sigma_{J^i_0}(w)\, dw=0$.
  \end{enumerate}
\end{proposition}

Our construction of the $\{J_0^i\}$ is similar to that of \cite[Proof
of Lemma 5.2]{Cha:2007-2}, but provides a stronger conclusion.  For
this purpose we need the following facts:

\begin{lemma}
  [Reparametrization Formula
  \cite{Litherland:1979-1,Cochran-Orr:1993-1,Cha-Ko:2000-1,Cha:2003-1}]
  \label{lemma:reparametrization-formula}
  For a knot $K$, let $K'$ be the $(r,1)$-cable of~$K$.  Then
  $\sigma_{K'}(w) = \sigma_{K}(w^r)$ for all $w\in S^1$ but finitely
  many points, and $\Delta_{K'}(t)=\Delta_K(t^r)$ up to
  multiplication by $at^n$, $a\ne 0$.
  \begin{trivlist}
  \item[] Addendum: If $w$ is a primitive $p^a$th root of unity for some
    prime $p$, then $\sigma_{K'}(w) = \sigma_{K}(w^r)$.
  \end{trivlist}
\end{lemma}  

The addendum is a consequence of the main statement, by the following
observation: if $w$ is a primitive $p^a$th root of unity, then for any
knot $K$, $\sigma_K$ is locally constant near $w$ since $w$ is not a
zero of the Alexander polynomial~$\Delta_K(t)$.  Therefore we can
perturb $w$ to avoid the exceptional points, without changing the
value of~$\sigma_K$.



\begin{proof}
  [Proof of Proposition~\ref{proposition:infection-knot-construction}]
  For each $i$, choose small $\epsilon_i>0$ such that the
  $\epsilon_i$-neighborhood $N_{\epsilon_i}(w_i)$ is disjoint from
  $T_i=\{w_j^r \mid j\le i$, $r\in \Z\}-\{w_i\}$.  Appealing to
  \cite{Cha-Livingston:2002-1}, \cite[Lemma 5.6]{Cha:2007-2}, there is
  a knot $K_i$ with a bump signature function supported by
  $N_{\epsilon_i}(w_i)\cup N_{\epsilon_i}(w_i^{-1})$, that is,
  $\sigma_{K_i}(w_i)> 0$ and $\sigma_{K_i}$ vanishes outside
  $N_{\epsilon_i}(w_i) \cup N_{\epsilon_i}(w_i^{-1})$.  (Recall the
  symmetry $\sigma_{K}(w)=\sigma_{K}(w^{-1})$.)

  Let $K_i'$ be the $(p_i,1)$-cable of $K_i$, and let $J_i=K_i
  \mathbin{\#} (-K_i')$.  By Reparametrization Formula, we have
  $\sigma_{J_i}(w_j^r) = \sigma_{K_i}(w_j^r) -
  \sigma_{K_i}(w_j^{rp_i})$.  Therefore, for $j=i$, we have
  $\sigma_{J_i}(w_i^r) = \sigma_{K_i}(w_i^r) - \sigma_{K_i}(1) =
  \sigma_{K_i}(w_i^r)$.  Thus $\sigma_{J_i}(w_i^r)\ne 0$ if and only
  if $r\equiv \pm 1 \text{ mod } p_i$.  Also, for $j<i$, we have
  $\sigma_{J_i}(w_j^r)=0$ since $\sigma_{K_i}(w_j^r)$ vanishes for
  any~$r$.

  By the Reparametrization Formula, we have $\int_{S^1}\sigma_{K_i}(w)\,
  dw = \int_{S^1}\sigma_{K_i'}(w)\, dw$.  It follows that
  $\int_{S^1}\sigma_{J_i}(w)\, dw = 0$.  Now the desired knot $J^i_0$
  is obtained by taking the connected sum of sufficiently large even
  number of copies of~$J_i$.
\end{proof}

By Proposition~\ref{proposition:solution-for-iterated-satellite} and
Proposition~\ref{proposition:infection-knot-construction} (C3) above,
we obtain immediately the conclusion that the associated $J^i_n$ are
$(n)$-solvable knots which have vanishing PTFA $L^2$-signature
obstructions.  (Note that we do not need (C1) and (C2) for this
purpose.)

The remaining part of this section is devoted to the proof that any
nontrivial linear combination $J=\mathop{\#}_i^{\vphantom{i}}
a_i^{\vphantom{i}} J^i_n$ is not integrally $(n.5)$-solvable.  This
consists of two parts.  In the first part, assuming that $J$ is
integrally $(n.5)$-solvable, we construct a 4-manifold $W_0$ from an
integral $(n.5)$-solution of $J$, which has $M(J^i_0)$ as a boundary
component for some fixed $i$ for which $a_i\ne 0$.
In the second part, using a non-PTFA mixed-coefficient commutator quotient of
$\pi_1(W_n)$ as a coefficient system, we apply our new obstruction
(Theorem~\ref{theorem:n.5-solvable-obstruction-amenable-in-D(R)}) in
order to derive a contradiction.  The first part is similar to
arguments in \cite{Cochran-Harvey-Leidy:2009-1} but we make several
technical simplifications by taking advantages of our new obstruction.
(See also Remark~\ref{remark:advantage-of-non-PTFA}.)

\subsubsection*{Proof of non-$(n.5)$-solvability, Part I: Construction
  of bounding 4-manifolds}

Suppose a (finite) linear combination $J=\mathop{\#}_i^{\vphantom{i}}
a_i^{\vphantom{i}} J^i_n$ is integrally $(n.5)$-solvable.  Dropping
terms with zero coefficient, we may assume $a_i\ne 0$ for each~$i$.
Furthermore we may assume that $a_1>0$ by taking concordance inverses
if necessary.  We use the following building blocks:

\begin{enumerate}
\item Let $V$ be an integral $(n.5)$-solution for $M(J)$.
\item Let $V_i$ be the $(n)$-solution of $M(J^i_n)$ given in
  Proposition~\ref{proposition:solution-for-iterated-satellite}.
\item Let $E_k$ be the standard cobordism from $M(J^1_k) \cup M(K_k)$
  to $M(J^1_{k+1})$, viewing $J^1_{k+1}$ as $K_k(\eta_k,J^1_k)$.
  Namely, $E_k = (M(K_k) \times[0,1]) \cup (M(J^1_k) \times [0,1])
  \mathbin{/} \sim$ where $N(\eta_k)$ in $M(K_k)=M(K_k)\times 0$ is
  identified with the solid torus $M(J^1_k)-X(J^1_k)$ in
  $M(J^1_k)=M(J^1_k)\times 0$ in such a way that a zero-linking
  longitude of $\eta_k$ is identified with a meridian of~$J^1_k$.  For
  more discussions on $E_k$, see
  \cite[p.~1429]{Cochran-Harvey-Leidy:2009-1}.
\item Let $C$ be the standard cobordism from $\bigcup_i a_i M(J^i_n)$
  to $M(J)$ obtained by glueing copies of $M(J^i_n)\times[0,1]$
  similarly to (3), viewing the connected sum as a satellite knot
  formed along meridian curves.  (Or, for an alternative description,
  see \cite[p.~113]{Cochran-Orr-Teichner:2002-1}.)
\end{enumerate}


Let $\epsilon_i=a_i/|a_i|$, $b_1=a_1-1$ ($\ge 0$)
and $b_i=|a_i|$ for $i\ge 2$.  For $r=1,\ldots, b_i$, let $V_{i,r}$
be a copy of $-\epsilon_i V_i$.
Define
\[
W_n=\Big(\bigcup_{i} \bigcup_{r=1}^{b_i} V_{i,r}\Big)
\amalgover{\bigcup_{i} b_i M(J^i_n)} C \amalgover{M(J)} V.
\]
Note that $\partial W_n=M(J^1_n)$.  For $k=n-1,n-2,\ldots,0$, define $W_k$ by
\begin{align*}
  W_k & = E_{k} \amalgover{M(J^1_{k+1})} E_{k+1} \amalgover{M(J^1_{k+2})}
  \cdots \amalgover{M(J^1_{n-1})} E_{n-1} \amalgover{M(J^1_{n})} W_{n}
  \\
  & = E_{k} \amalgover{M(J^1_{k+1})} W_{k+1}.
\end{align*}
We have $\partial W_k=M(J^1_k) \cup M(K_k) \cup \cdots \cup
M(K_{n-1})$ for $k<n$.  See Figure~\ref{figure:cobordism-W_k}.

\begin{figure}[h]
  \labellist
  \small\hair 0mm
  \pinlabel {$V$} at 160 15
  \pinlabel {$C$} at 160 60
  \pinlabel {$M(J)$} [r] at -2 40
  \pinlabel {$M(J^1_{n})$} [r] at -2 80
  \pinlabel {$M(J^1_{n-1})$} [r] at -2 120
  \pinlabel {$M(J^1_{k})$} [r] at -2 187
  \pinlabel {$M(K_k)$} [l] at 48 187
  \pinlabel {$M(K_{n-1})$} [l] at 100 130
  \pinlabel {$V_{i,r}$}  at 177 100
  \pinlabel {$V_{i,r}$}  at 290 100
  \pinlabel {$E_{n-1}$}  at 60 100
  \pinlabel {$E_{n-2}$}  at 38 138
  \pinlabel {$M(J^i_n)$} [r] at 258 83
  \endlabellist
  \begin{center}
    \noindent\kern4em\includegraphics[scale=.9]{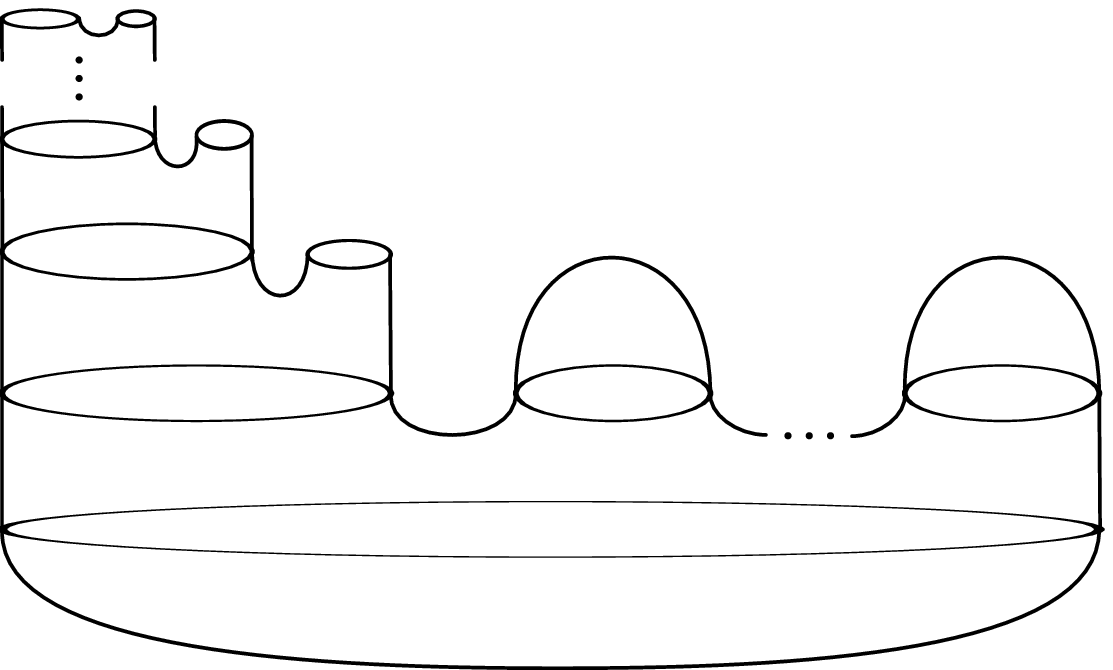}
    \caption{The cobordism $W_k$}
    \label{figure:cobordism-W_k}
  \end{center}
\end{figure}

From now on, we denote by $\cP$ the sequence $(R_0,R_1,\ldots,R_{n})$
where $R_k=\Q$ for $k< n$, $R_{n}=\Z_{p_1}$.  Then the mixed-coefficient
commutator series $\cP^k G$ of a group $G$ is defined for
$k=0,1,\ldots,n+1$ as in
Definition~\ref{definition:mixed-coefficient-commutator-series}.  We have
$G^{(k)} \subset \cP^k G$.

\begin{theorem}
  \label{theorem:nontriviality-of-coefficient-system}
  For $k=0,1,\ldots, n$, the homomorphism
  \[
  \phi_k\colon \pi_1(W_k) \to \pi_1(W_k)/\cP^{n-k+1}\pi_1(W_k).
  \]
  sends a meridian of $J^1_k$ into the abelian subgroup
  $\cP^{n-k}\pi_1(W_k)/\cP^{n-k+1}\pi_1(W_k)$.  Furthermore, the image
  of a meridian of $J^1_k$ under $\phi_k$ has order $p_1$ if $k=0$,
  and has order $\infty$ if $k>0$.
\end{theorem}

An immediate consequence of the first statement is that $\phi_k$
restricted on $\pi_1 M(J^1_k)$ is into
$\cP^{n-k}\pi_1(W_k)/\cP^{n-k+1}\pi_1(W_k)$.  Roughly speaking, the
first statement is essentially on the same lines with the arguments of
Lemma~\ref{lemma:infection-curve-depth-for-iterated-satellite} that
$[\eta_k]$ (which is identified with the meridian of $J_k$) is in the
$(n-k)$th term of the derived series.  The rational derived series
analogue of the second statement (for a similar but more complicated
4-manifold) was shown in
\cite[Section~8]{Cochran-Harvey-Leidy:2009-1}.  Our argument is
similar to their proof, but in order to handle the mixed-coefficient
commutator series which gives our non-PTFA coefficient systems, we
need to use a new ingredient, namely modulo $p$ higher-order
Blanchfield pairings.  We postpone related discussions and the proof
of Theorem~\ref{theorem:nontriviality-of-coefficient-system} to the
next section.

\subsubsection*{Proof of non-$(n.5)$-solvability, Part II: Computation
  of $L^2$-signature}

Now we consider the coefficient system $\phi_0\colon \pi_1(W_0) \to
G=\pi_1(W_0)/\cP^{n+1}\pi_1(W_0)$.  As an abuse of notation, we denote
by $\phi_0$ various restrictions of $\phi_0$, particularly the
restriction on~$\pi_1M(J^1_0)$.  Note that $G^{(n+1)}=\{e\}$ obviously,
and $G$ is amenable and lies in $D(\Z_{p_1})$ by
Lemma~\ref{lemma:mixed-coefficient-commutator-series-is-amenable-in-D(R)}.

For convenience, for a 4-manifold $X$ over $G$, we temporarily denote
the $L^2$-signature defect by $S_G(X)=\lsign_G(X)-\sign(X)$.  Since
$\partial W_0$ is the disjoint union of $M(J^1_0)$,
$M(K_0),\ldots,M(K_{n-1})$, we have
\[
S_G(W_0) = \rhot(M(J^1_0),\phi_0) + \sum_{k=0}^{n-1} \rhot(M(K_k),\phi_0)
\]
By Novikov additivity, we have
\[
S_G(W_0) = S_G(V) + S_G(C) + \sum_{i} \sum_{r=1}^{b_i}
S_G(V_{i,r}) + \sum_{k=0}^{n-1} S_G(E_k)
\]
Now we investigate the terms in right hand sides of the above two
equations:

\begin{enumerate}
\item $\rhot(M(J^1_0),\phi_0)=\sum_{r=0}^{p_1-1} \sigma_{J^1_0}(w_1^r)$.
  To verify this, note that the image of a meridian of $J^1_0$ under
  $\phi_0$ has order~$p_1$ by
  Theorem~\ref{theorem:nontriviality-of-coefficient-system}.  By
  Lemma~\ref{lemma:abelian-rho-invariant}, the claim follows.
\item $S_G(V)= 0$, $S_G(C)=0$, $S_G(E_k)=0$.  First, since $V$ is an
  integral $(n.5)$-solution for $M(J)$, $S_G(V)=\rhot(M(J),\phi_0)$
  vanishes by
  Theorem~\ref{theorem:n.5-solvable-obstruction-amenable-in-D(R)}.
  For the latter two claims, appeal to
  \cite[Lemma~2.4]{Cochran-Harvey-Leidy:2009-1} (see also
  \cite[Lemma~4.2]{Cochran-Orr-Teichner:2002-1}).
\item $S_G(V_{1,r})$ is either 0 or $-\sum_{r=0}^{p_1-1}
  \sigma_{J^1_0}(w_1^r)$, and $S_G(V_{i,r})=0$ for $i\ge 2$.  To show
  this, first apply
  Proposition~\ref{proposition:solution-for-iterated-satellite} to
  obtain
  \[
  S_G(V_{i,r})=-\epsilon_i\rhot(M(J_n^i),\phi_0)=-\epsilon_i
  \rhot(M(J_0^i),\alpha_{i,r})
  \]
  for some $\alpha_{i,r}\colon \pi_1 M(J_0^i) \to \Z_d$ where $d$ is
  the order of an element in~$G^{(n)}$.  Recall that $G^{(n)}$ is a
  subgroup of $\cP^{n}\pi_1(W_0)/\cP^{n+1}\pi_1(W_0)$ which is a
  vector space over~$\Z_{p_1}$.  It follows that $d$ is either $1$
  or~$p_1$.  Therefore $\rhot(M(J^i_0),\alpha_{i,r})$ is equal to
  either 0 or $\sum_{r=0}^{p_1-1} \sigma_{J^i_0}(w_1^r)$ by
  Lemma~\ref{lemma:abelian-rho-invariant}.  For $i=1$, since
  $\epsilon_1=1$, we obtain the conclusion.  For $i\ge 2$, the sum is
  zero since each summand vanishes by the condition~(C2).
\end{enumerate}

From the above facts, it follows that
\[
C\cdot \sum_{r=0}^{p_1-1} \sigma_{J^1_0}(w_1^r) = -\sum_{k=0}^{n-1}
\rhot(M(K_k),\phi_0)
\]
where $C\ge 1$.
But it contradicts the condition~(C1).  This shows that the linear
combination $J=\mathop{\#}_i^{\vphantom{i}}
  a_i^{\vphantom{i}} J^i_n$ is not integrally $(n.5)$-solvable.

\begin{remark}
  \label{remark:advantage-of-non-PTFA}
  In our construction above, the only requirement for the slice knots
  $K_k$ is that their classical Alexander modules are nontrivial and
  generated by~$\eta_k$.  Interestingly, while our examples are
  subtler in the knot concordance group than prior examples of
  \cite{Cochran-Harvey-Leidy:2009-1} which are constructed in a
  similar way but detected by PTFA $L^2$-signatures, the construction
  in \cite{Cochran-Harvey-Leidy:2009-1} need an additional significant
  restriction that certain $\rho$-invariants of the $M(K_k)$ mush
  vanish; see Step 1 of \cite[Theorem
  8.1]{Cochran-Harvey-Leidy:2009-1}.  This is essentially because our
  $L^2$-signatures are sharp enough to separate the contributions of
  the knots $J^i_0$ ($i>1$) from those of~$J^1_0$; see the above
  analysis of the terms $S_G(V_{i,r})$.  In case of PTFA
  $L^2$-signatures, it is unable to separate these contributions, and
  hence \cite{Cochran-Harvey-Leidy:2009-1} constructs more complicated
  4-manifolds (using more pieces called ``$\mathcal{R}$-caps'') and
  carries out a clever technical analysis using the additional
  assumption on the vanishing of certain $\rho$-invariants
  of~$M(K_k)$.  Our $L^2$-signature obstruction enables us not only to
  produce subtler examples but also to simplify such complications.
\end{remark}

\section{Modulo $p$ noncommutative Blanchfield pairing}
\label{section:mod-p-Blanchfield-duality}

In this section we introduce the modulo $p$ higher-order
noncommutative Blanchfield pairing on a 3-manifold $M$ (e.g., the
zero-surgery manifold of a knot).  We remark that a noncommutative
Blanchfield pairing first appeared in Duval's work~\cite{Duval:1986-1}
on boundary links, as a linking form over the group ring of a free
group.  Cochran-Orr-Teichner first introduced a noncommutative
Blanchfield pairing of knots, over the group ring $\Z G$ (and its
localizations) for PTFA groups~$G$~\cite{Cochran-Orr-Teichner:1999-1}.
For further development of the latter, see also work of
Cochran~\cite{Cochran:2002-1} and Leidy~\cite{Leidy:2006-1}.

In the first subsection, we define the modulo $p$ noncommutative
Blanchfield pairing and discuss modulo $p$ generalizations of various
known properties of the prior noncommutative Blanchfield pairings.  In
the second subsection, we discuss how the modulo $p$ version is
combined with techniques of
Cochran-Harvey-Leidy~\cite{Cochran-Harvey-Leidy:2009-1} to show the
nontriviality of certain non-PTFA coefficient systems obtained as
mixed-coefficient commutator series quotients, as asserted in
Theorem~\ref{theorem:nontriviality-of-coefficient-system}.

\subsection{Definition and properties}

Suppose $R$ is a fixed commutative ring with unity and $M$ is a closed
3-manifold endowed with a homomorphism $\phi\colon \pi_1(M) \to G$
satisfying the following:

\begin{enumerate}
\item[(BL1)] $RG$ is an Ore domain, namely, the quotient skew-field
  $\K=RG (RG-\{0\})^{-1}$ is well-defined.
\item[(BL2)] $H_1(M;\K)$ vanishes.
\end{enumerate}

Assuming these conditions, the noncommutative Blanchfield pairing can
be defined, following~\cite{Cochran-Orr-Teichner:1999-1}. We outline
the definition: from the short exact sequence $0 \to RG \to \K \to
\K/RG \to 0$, we obtain a Bockstein homomorphism
\[
H^1(M;\K/RG) \xrightarrow{\cong} H^2(M;RG)
\]
which is an isomorphism since $H^2(M;\K) = H_1(M;\K)=0$ by duality and
(BL2) and since $H^1(M;\K)=\Hom_\K(H_1(M;\K),\K)=0$ by the universal
coefficient theorem over the skew field~$\K$.  Composing the inverse
of the Bockstein homomorphism with the Poincare duality isomorphism
and the Kronecker evaluation map, we obtain the following:

\begin{proposition}[Essentially due to~\cite{Cochran-Orr-Teichner:1999-1}]
  \label{proposition:definition-Blanchfield-form}
  Under the assumptions (BL1), (BL2), there is a linking form
  \begin{multline*}
    \def\to{\rightarrow}
    \Bl\colon \overline{H_1(M;RG)} \to H^2(M;R G) \to H^1(M;\K/RG)
    \to \Hom_{R G} (H_1(M;R G), \K/R G).
  \end{multline*}
\end{proposition}

We often denote $\Bl(x)(y)$ by $\Bl(x,y)$, viewing it as a sesquilinear
pairing.  It can be shown that $\Bl$ is hermitian, i.e.,
$\Bl(x,y)=\overline{\Bl(y,x)}$~\cite{Cochran-Orr-Teichner:1999-1}.

It is known that (BL1) is satisfied for $R=\Q$ (or $\Z$) and any PTFA
group~$G$.  This case has been used in several recent work, most
notably in \cite{Cochran-Harvey-Leidy:2009-1}.  The following
observation enables us to define $\Bl$ when $R=\Z_p$ ($p$ prime) as
well:

\begin{lemma}
  \label{lemma:mod-p-PTFA-group-ring-is-ore-domain}
  If $G$ is PTFA, then $\Z_p G$ is an Ore domain.  Consequently, there
  exists the Ore localization $\K=\Z_pG (\Z_pG-\{0\})^{-1}$, which is
  a skew-field containing $\Z_p G$.
\end{lemma}

\begin{proof}
  Since $G$ is PTFA and $\Z_p$ is a field, $\Z_pG$ has no zero divisor
  by Bovdi's theorem (see \cite[p.~592]{Passman:1977-1}).  Since $G$
  is solvable, it follows that $\Z_pG$ is an Ore domain, by Lewin's
  result~\cite{Lewin:1972-1} (see also \cite[p.~611]{Passman:1977-1}).
\end{proof}

The main example to keep in mind is: suppose $M$ is the zero-surgery
manifold of a knot $K$ in~$S^3$ endowed with a nontrivial PTFA
coefficient system $\phi\colon \pi_1(M) \to G$.  Then, it is known
that (BL1) and (BL2) are satisfied for
$R=\Q$~\cite{Cochran-Orr-Teichner:1999-1}.  For $R=\Z_p$, (BL1) is
satisfied by Lemma~\ref{lemma:mod-p-PTFA-group-ring-is-ore-domain},
and (BL2) is satisfied by appealing to the following mod $p$ analogue
of \cite[Proposition~2.11]{Cochran-Orr-Teichner:1999-1}.  Its proof is
identical with that of \cite[Proposition
2.11]{Cochran-Orr-Teichner:1999-1} and therefore omitted.

\begin{lemma}
  \label{lemma:1st-betti-number-over-skew-field}
  Suppose $G$ is PTFA, $R=\Z_p$ or $\Q$, $\K$ is the quotient
  skew-field of $RG$, $X$ is a finite CW complex, and $\phi\colon
  \pi_1(X) \to G$ is a nontrivial homomorphism.  Then $\dim_\K
  H_0(X;\K)=0$ and $\dim_\K H_1(X;\K) \le \dim_{R} H_1(X;R) -1$.
\end{lemma}

If $RG$ is a PID, then $\K/RG$ is divisible and so injective
over~$RG$.  Thus the evaluation map in the definition of $\Bl$ is an
isomorphism.  It follows that $\Bl$ is nonsingular.  For example, this
applies to the case of $G=\Z$ and $R=\Q$ or~$\Z_p$.



Many prior known results on the noncommutative Blanchfield pairing
over $\Q$ also hold for our new setup, namely for $R=\Z_p$.  In what
follows we discuss some of such results
(Theorems~\ref{theorem:Blanchfield-change-of-coefficients}
and~\ref{theorem:self-annihilating-submodule-blanchfield-form}) that
we need.  For these results, since the proofs for $R=\Q$ in the
literature can be carried out for $R=\Z_p$ as well (in many cases we
need Lemma~\ref{lemma:1st-betti-number-over-skew-field} instead of
\cite[Proposition 2.11]{Cochran-Orr-Teichner:1999-1}), we will give
references for the $R=\Q$ case, without proofs for the $R=\Z_p$ case.

First, the following result concerns the effect of coefficient change.

\begin{theorem}
  [{\cite[Theorem~4.7]{Leidy:2006-1},
    \cite[Theorem~5.16]{Cha:2003-1}, \cite[Lemma~6.5,
    Theorem~6.6]{Cochran-Harvey-Leidy:2009-1} for $R=\Q$}]
  \label{theorem:Blanchfield-change-of-coefficients}
  Suppose $R=\Q$ or $\Z_p$, $\phi\colon\pi_1(M)\to G$ is a
  homomorphism and $\iota\colon G\to \Gamma$ is an injection such that
  both $(M,\phi)$ and $(M,\iota\phi)$ satisfy (BL1, BL2).  Let $\K$
  and $\K'$ be the quotient skew-fields of $RG$ and $R\Gamma$, and
  $\Bl$ and $\Bl'$ be the Blanchfield forms on $H_1(M;RG)$ and
  $H_1(M;R\Gamma)$, respectively.  If $RG$ is a PID, then
  \begin{enumerate}
  \item The map $\iota$ induces an injection $\iota_*\colon \K/RG \to
    \K'/R\Gamma$.
  \item $H_1(M;R\Gamma) = R\Gamma\otimes_{RG} H_1(M;RG)$ and
    $\Bl'(1\otimes x, 1\otimes y)=\iota_*\Bl(x,y)$.
  \end{enumerate}
\end{theorem}


Next we introduce an $R$-coefficient homology version of
Cochran-Harvey-Leidy's notion of a rational $(n)$-bordism
\cite[Definition 5.2]{Cochran-Harvey-Leidy:2009-1}.

\begin{definition}
  \label{definition:integral-n-bordism}
  For $R=\Q$ or $\Z_p$, a compact 4-manifold $W$ with $\pi=\pi_1(W)$
  is called an \emph{$R$-coefficient $(n)$-bordism} if there are
  elements $x_1,\ldots,x_r,y_1,\ldots,y_r \in H_2(W;R[\pi/\pi^{(n)}])$
  such that $2r=\dim_R \Coker\{H_2(\partial W;R)\to H_2(W;R)\}$ and
  the $R[\pi/\pi^{(n)}]$-coefficient intersection form $\lambda_n$ on
  $H_2(W;R[\pi/\pi^{(n)}])$ satisfies $\lambda_n(x_i,x_j)=0$ and
  $\lambda_n(x_i,y_j)=\delta_{ij}$.  We call $\{x_i\}$ and $\{y_j\}$
  an \emph{$(n)$-langrangian} and its \emph{$(n)$-dual}, respectively.
\end{definition}



\begin{theorem}
  [{\cite[Section 5, 6,
    Theorem~6.3]{Cochran-Harvey-Leidy:2009-1} for $R=\Q$}]
  \label{theorem:self-annihilating-submodule-blanchfield-form}
  Suppose $R=\Q$ or $\Z_p$, $W$ is an $R$-coefficient $(n)$-bordism
  and each boundary component $\partial_i W$ of $W$ satisfies
  $b_1(\partial_i W;R)=1$.  Suppose $M$ is a boundary component of $W$
  endowed with a nontrivial homomorphism $\phi\colon\pi_1(M)\to G$
  into a PTFA group $G$, so that the Blanchfield pairing $\Bl$ on
  $H_1(M;RG)$ is defined.  If $\phi$ extends to $\pi_1(W)$ and
  $G^{(n)}=\{e\}$, then $P=\Ker\{ H_1(M;RG) \to H_1(W;RG)\}$ satisfies
  $\Bl(P,P)=0$.
\end{theorem}


Obviously an integral $(n)$-solution is an integral ($\Z$-coefficient)
$(n)$-bordism.  We remark that an integral $(n)$-bordism is not
necessarily a $\Z_p$-coefficient $(n)$-bordism, since the $\Z_p$-rank
of $H_*(-;\Z_p)$ may be greater than the $\Z$-rank of $H_*(-;\Z)$.
However, the following observation says that, for example, an integral
$(n)$-solution for a knot in $S^3$ is a $\Z_p$-coefficient
$(n)$-bordism for any~$p$.

\begin{lemma}
  \label{lemma:integral-solution-is-mod-p-bordism}
  Suppose $W$ is an integral $(n)$-solution for a 3-manifold $M$ and
  $H_1(M)$ has no $p$-torsion.  Then
  \[
  \dim_{\Z_p} \Coker\{H_2(M;\Z_p)\rightarrow H_2(W;\Z_p)\} =
  b_2(W;\Z_p)=b_2(W)
  \]
  for any~$p$.  Consequently, $W$ is a $\Z_p$-coefficient $(n)$-bordism.
\end{lemma}

\begin{proof}
  $H_2(M;\Z_p)\rightarrow H_2(W;\Z_p)$ is equal to
  $H^1(M;\Z_p)\rightarrow H^2(W,M;\Z_p)$ by duality.  The latter is a
  zero map by universal coefficient, since $H_1(M) \to H_1(W)$ is an
  isomorphism.  From this the first equality follows.  The second
  equality follows from
  Lemma~\ref{lemma:mod-p-betti-number-of-solution}.
\end{proof}

\subsection{Analysis of mixed-coefficient commutator quotient coefficient systems}


In this subsection we give a proof of the following assertion used in
Section~\ref{subsection:non-solvability-of-examples}.  We use the
notations in Section~\ref{subsection:non-solvability-of-examples}, and
assume the hypothesis of
Theorem~\ref{theorem:linearly-independent-knots}.  For simplicity, we
denote $J^1_k$ by~$J_k$ and denote $p_1$ by~$p$.

\theoremstyle{plain}
\newtheorem*{tempthm}{Theorem~\ref{theorem:nontriviality-of-coefficient-system}}
\begin{tempthm}
  For $k=0,1,\ldots, n$, the homomorphism
  \[
  \phi_k\colon \pi_1(W_k) \to \pi_1(W_k)/\cP^{n-k+1}\pi_1(W_k).
  \]
  sends a meridian of $J_k$ into the abelian subgroup
  $\cP^{n-k}\pi_1(W_k)/\cP^{n-k+1}\pi_1(W_k)$.  Furthermore, the image
  of a meridian of $J_k$ under $\phi_k$ has order $p$ if $k=0$, and
  has order $\infty$ if $k>0$.
\end{tempthm}

We denote the meridian of $J_k$ and $K_k$ in $M(J_k)$ and $M(K_k)$ by
$\mu_k$ and $\nu_k$, respectively.  Recall that $E_k$ is the standard
cobordism between $M(J_k)\cup M(K_k)$ and~$M(J_{k+1})$ (see Section
\ref{subsection:non-solvability-of-examples}, Proof of
non-$(n.5)$-solvability).  Then $\nu_k\subset M(K_k) \subset E_k$ and
$\mu_{k+1} \subset M(J_{k+1}) \subset E_k$ are isotopic in $E_k$, and
$\mu_k\subset M(J_k)\subset E_k$ and~$\eta_k\subset M(K_k)\subset
E_k$ are isotopic in~$E_k$.

\begin{assertion}
  \label{assertion:derived-series-quotients-of-W_k}
  Suppose $\eta_k$ has linking number zero with $K_k$ for any~$k$.
  Then the inclusion $W_k \to W_{k+1}$ gives rise to an isomorphism
  \[
  \pi_1(W_k)/\pi_1(W_k)^{(n-k+1)} \to
  \pi_1(W_{k+1})/\pi_1(W_{k+1})^{(n-k+1)}.
  \]
  for any~$k<n$.  Consequently, the induced map
  \[
  \cP^{n-k}\pi_1(W_k)/\cP^{n-k+1}\pi_1(W_k) \to
  \cP^{n-k}\pi_1(W_{k+1})/\cP^{n-k+1}\pi_1(W_{k+1})
  \]
  is an isomorphism.
\end{assertion}

To prove this, let $X$ be a copy of $X(J_k)$ and view it as a subspace
of
\[
X \amalg_{\partial X} (M(K_k)-N(\eta_k))=M(J_{k+1})\subset W_{k+1}.
\]
Denote a meridian and longitude of $J_k$ on $\partial X$ by $\mu$ and
$\lambda$, respectively.  From the construction of $E_k$, it follows
that $\big( M(J_{k+1}) \cup (S^1 \times D^2) \big)
\mathbin{/} \partial X\sim S^1\times S^1$ is a deformation retract
of~$E_k$, where $\mu$ and $\lambda$ are identified with $S^1\times *$
and $*\times S^1$, respectively.  Therefore $W_k=E_k
\amalg_{M(J_{k+1})} W_{k+1}$ is homotopy equivalent to $\big(W_{k+1}
\cup (S^1 \times D^2)\big) \mathbin{/} \partial X \sim S^1\times S^1$.
It follows that
\[
\pi_1(W_k)/\pi_1(W_k)^{(n-k+1)} \cong \pi_1(W_{k+1}) / \langle
\pi_1(W_{k+1})^{(n-k+1)}, \lambda\rangle
\]
where $\langle - \rangle$ denotes the normal subgroup in
$\pi_1(W_{k+1})$ generated by~$-$.  Since $[\lambda]$ is in
$\pi_1(X)^{(2)}\subset \langle \mu\rangle^{(2)}$, it suffices to show
$[\mu]\in \pi_1(W_{k+1})^{(n-k-1)}$.  If $k=n-1$, it holds obviously
since $\pi_1(W_{k+1})^{(n-k-1)}=\pi_1(W_{k+1})$.  Suppose $k\le n-2$.
Since $\mu$ lies in $M(J_{k+1})\subset W_{k+1}$, $[\mu]\in \langle
\mu_{k+1}\rangle$.  Since $\eta_{k+1}$ has linking number zero with
$K_{k+1}$, we have, in $\pi_1(W_{k+1})$, $\langle\mu_{k+1}\rangle =
\langle\eta_{k+1}\rangle = \langle\nu_{k+1}\rangle^{(1)} =
\langle\mu_{k+2}\rangle^{(1)}$.  Applying it repeatedly, it follows
that $[\mu] \in \langle\mu_{k+1}\rangle \subset
\langle\mu_n\rangle^{(n-k-1)}$.  This proves the first conclusion of
Assertion~\ref{assertion:derived-series-quotients-of-W_k}.  The second
conclusion of
Assertion~\ref{assertion:derived-series-quotients-of-W_k} follows
since $\cP^iG$ is a characteristic subgroup of $G$
containing~$G^{(i)}$.

Also, in the above argument, we have observed that the meridian
$\mu_k$ in $M(J_k)\subset W_k$ lies in $\pi_1(W_k)^{(n-k)}\subset
\cP^{n-k} \pi_1(W_k)$.  From this the first conclusion of
Theorem~\ref{theorem:nontriviality-of-coefficient-system} follows.

\begin{assertion}
  \label{assertion:W_k-is-n-bordism}
  For $R=\Q$ or $\Z_p$ and for any $k$, $W_k$ is an $R$-coefficient
  $(n)$-bordism.
\end{assertion}

By straightforward Mayer-Vietoris arguments applied to our
construction of $W_k$, one can show $\Coker\{H_2(\partial W_k;R) \to
H_2(W_k;R)\}\cong H_2(V;R)\oplus \bigoplus_{i,j} H_2(V_{i,r};R)$.
(See, for a similar argument for a more complicated 4-manifold,
\cite[Proof of Proposition 8.2]{Cochran-Harvey-Leidy:2009-1}.)  Since
$W_n$ and $V_{i,r}$ are integral $(n)$-solutions of knots, $W_n$ and
the $V_{i,r}$ are $R$-coefficient $(n)$-bordisms by
Lemma~\ref{lemma:integral-solution-is-mod-p-bordism}.  By the above
$H_2(-;R)$ computation, it follows (the images of) the
$(n)$-lagrangians and $(n)$-duals of $W_n$ and the $V_{i,r}$ form an
$(n)$-langrangian and $(n)$-dual for~$W_k$.  This proves
Assertion~\ref{assertion:W_k-is-n-bordism}.

Now we use an induction on $k=n, n-1, \ldots, 0$ to prove that the
order of the image of a meridian of $J_k$ under
\[
\phi_k\colon \pi_1(W_k)\to \pi_1(W_k)/\cP^{n-k+1}\pi_1(W_k)
\]
is $p$ if $k=0$, $\infty$ otherwise.  For $k=n$, since $W_n$ is a
solution, we have $H_1(M(J_n)) \cong H_1(W_n) =
H_1(W_n)/\text{torsion} = \pi_1(W_n)/\cP^{1}\pi_1(W_n)$.  The
conclusion follows from this.  

Now, assuming that it holds for $k+1$, we will show the conclusion
for~$k$.  To simplify notations, in this proof we temporarily denote
$\pi=\pi_1(W_{k+1})$, $G=\pi/\cP^{n-k}\pi$.

%
By the first part of
Theorem~\ref{theorem:nontriviality-of-coefficient-system}, the
coefficient system $\pi_1(M_{J_{k+1}})\to G$ factors through the
abelian group $\cP^{n-k-1}\pi/\cP^{n-k}\pi$, and so it factors through
$H_1(M_{J_{k+1}})$.  By the induction hypothesis, the meridian
$[\mu_{k+1}]$ has infinite order in~$G$.  Thus, $H_1(M(J_{k+1}))\cong
\langle t|\cdot\rangle$ actually injects into~$G$.  Therefore, by
Theorem~\ref{theorem:Blanchfield-change-of-coefficients},
\[
  H_1(M(J_{k+1});R G) =
  RG \otimesover{R[t^{\pm1}]} H_1(M(J_{k+1});R[t^{\pm1}])
  = RG \otimesover{R[t^{\pm1}]} H_1(M(K_k);R[t^{\pm1}])
\]
Consider the following commutative diagram, where $R=R_{n-k}$:
\[
\xymatrix@C=1.5em@R=2em{
  H_1(M(J_{k+1});\Z G) \ar[r]\ar[d] &
  H_1(W_{k+1};\Z G) \ar[r]\ar[d] &
  \dfrac{\cP^{n-k}\pi}{\cP^{n-k+1}\pi} \ar@{_{(}->}[dl]
  \\
  H_1(M(J_{k+1});R G ) \ar[r] &
  H_1(W_{k+1};R G )
}
\]
We claim that for $[\eta_k]\in H_1(M(K_k);\Z[t^{\pm1}])$, the image of
$1\otimes [\eta_k]\in H_1(M(J_{k+1});\Z G)$ in $H_1(W_{k+1};RG)$ is
nontrivial.  Suppose not.  Then, by
Theorem~\ref{theorem:self-annihilating-submodule-blanchfield-form},
the noncommutative Blanchfield form on $H_1(M(J_{k+1});R G)$ vanishes
on $(1\otimes [\eta_k], 1\otimes [\eta_k])$.  By
Theorem~\ref{theorem:Blanchfield-change-of-coefficients}, it follows
that the classical Blanchfield form on
$H_1(M(J_{k+1});R[t^{\pm1}])=H_1(M(K_k);R[t^{\pm1}])$ vanishes on
$([\eta_k],[\eta_k])$.  But this is a contradiction since $[\eta_k]$
generates the nonzero module $H_1(M(K_k);R[t^{\pm1}])$ and the
classical Blanchfield form of a knot is always nonsingular (for both
$R=\Z_p$ and $\Q$).  Recall that, for $k=0$, we have $R_0=\Z_p$ and
$H_1(M(K_0);\Z_p[t^{\pm1}])\ne 0$ since the Alexander polynomial
$\Delta_{K_0}(t)$ is not a unit even in~$\Z_p[t^{\pm1}]$, by our
choice of the primes~$p_i$.  This proves the claim.

From the claim, it follows that the image of $[\eta_k]$ in
$\cP^{n-k}\pi/\cP^{n-k+1} \pi$ is nontrivial.  By
Assertion~\ref{assertion:derived-series-quotients-of-W_k}, it follows
that the image of $[\mu_k] = [\eta_k]$ in
\[
\frac{\cP^{n-k}\pi_1(W_k)}{\cP^{n-k+1}\pi_1(W_k)} \cong
\frac{\cP^{n-k}\pi}{\cP^{n-k+1} \pi}
\]
is nontrivial.  Since this is a torsion-free abelian group for $k\ne
0$ and a vector space over $\Z_{p}$ for $k=0$ by our choice of $\cP$,
the order of the image of $[\mu_k]$ is $p$ if $k=0$, and $\infty$
otherwise.  This completes the proof of
Theorem~\ref{theorem:nontriviality-of-coefficient-system}.

\bibliographystyle{amsalpha} \renewcommand{\MR}[1]{}

\bibliography{research}

\end{document}